\documentclass[11pt]{amsart}
\usepackage{mathrsfs}
\usepackage{cases}
\usepackage{amssymb}
\usepackage{amsfonts}
\usepackage{amssymb,amsfonts}

\usepackage{amsmath,amssymb,amsthm,hyperref}
\allowdisplaybreaks

\begin{document}
\theoremstyle{plain}
\newtheorem{thm}{Theorem}[section]
\newtheorem{theorem}[thm]{Theorem}
\newtheorem{lemma}[thm]{Lemma}
\newtheorem{corollary}[thm]{Corollary}
\newtheorem{corollary*}[thm]{Corollary*}
\newtheorem{proposition}[thm]{Proposition}
\newtheorem{proposition*}[thm]{Proposition*}
\newtheorem{conjecture}[thm]{Conjecture}
\theoremstyle{definition}
\newtheorem{construction}{Construction}
\newtheorem{notations}[thm]{Notations}
\newtheorem{question}[thm]{Question}
\newtheorem{problem}[thm]{Problem}
\newtheorem{remark}[thm]{Remark}
\newtheorem{remarks}[thm]{Remarks}
\newtheorem{definition}[thm]{Definition}
\newtheorem{claim}[thm]{Claim}
\newtheorem{assumption}[thm]{Assumption}
\newtheorem{assumptions}[thm]{Assumptions}
\newtheorem{properties}[thm]{Properties}
\newtheorem{example}[thm]{Example}
\newtheorem{comments}[thm]{Comments}
\newtheorem{blank}[thm]{}
\newtheorem{observation}[thm]{Observation}
\newtheorem{defn-thm}[thm]{Definition-Theorem}

\newcommand{\sM}{{\mathcal M}}


\title{Integrality structures in topological strings I:\\ framed unknot}
       \author{Wei Luo, Shengmao Zhu}
        \address{Department of Mathematics and Center of Mathematical Sciences, Zhejiang University, Hangzhou, Zhejiang 310027, China}
        \email{luowei@cms.zju.edu.cn, szhu@zju.edu.cn}

\begin{abstract}
We study the open string integrality invariants (LMOV invariants)
for toric Calabi-Yau 3-folds with Aganagic-Vafa brane (AV-brane). In
this paper, we focus on the case of the resolved conifold with one
out AV-brane in any integer framing $\tau$, which is the large $N$
duality of the Chern-Simons theory for a framed unknot with integer
framing $\tau$ in $S^3$. We compute the explicit formulas for the
LMOV invariants in genus $g=0$ with any number of holes, and prove
their integrality. For the higher genus LMOV invariants with one
hole, they are reformulated into a generating function $g_{m}(q,a)$,
and we prove that $g_{m}(q,a)\in
(q^{1/2}-q^{-1/2})^{-2}\mathbb{Z}[(q^{1/2}-q^{-1/2})^2,a^{\pm 1/2}]$
for any integer $m\geq 1$.  As a by product, we compute the reduced
open string partition function of $\mathbb{C}^3$ with one AV-brane
in framing $\tau$. We find that, for $\tau\leq -1$, this open string
partition function is equivalent to the Hilbert-Poincar\'e series of
the Cohomological Hall algebra of the $|\tau|$-loop quiver. It gives
an open string GW/DT correspondence.
\end{abstract}
\maketitle \tableofcontents

\section{Introduction}
We study the integrality structures in topological string theory.
Let $X$ be a Calabi-Yau 3-fold, by the work of Gopakumar and Vafa
\cite{GV1}, the closed string free energy $F^X$, which is the
generating function of Gromov-Witten invariants $K_{g,Q}$, has the
following structure:
\begin{align*}
F^X=\sum_{g\geq 0}g_s^{2g-2}\sum_{Q\neq 0}K_{g,Q}e^{-Q
\cdot\omega}=\sum_{g\geq 0, d\geq 1}\sum_{Q\neq
0}\frac{1}{d}N_{g,Q}\left(2\sin\frac{dg_s}{2}\right)^{2g-2}e^{-dQ\cdot
\omega}
\end{align*}
where $N_{g,Q}$ are integers and vanish for large $g$ or $Q$. When
$X$ is a toric Calabi-Yau 3-fold, the above Gopakumar-Vafa
conjecture was proved in \cite{Peng,Konishi}.

In the open string case, let us consider a Calabi-Yau 3-fold $X$
with a Lagrange submanifold $\mathcal{D}$ in it. According to the
work of  Ooguri and Vafa \cite{OV}, the generating function of the
open Gromov-Witten invariants can also be expressed in terms of a
series of new integers which were refined by Labastida, Mari\~no and
Vafa \cite{LM1,LM2,LMV}:
\begin{align} \label{openmultiplecoveringintro}
&\sum_{g\geq 0}\sum_{Q\neq 0}g_s^{2g-2+l(\mu)}K_{\mu,g,Q}e^{-Q\cdot
\omega}\\\nonumber &=\sum_{g\geq 0}\sum_{Q\neq
0}\sum_{d|\mu}\frac{(-1)^{l(\mu)+g}}{\prod_{i=1}^{l(\mu)}\mu_i}d^{l(\mu)-1}n_{\mu/d,g,Q}\prod_{j=1}^{l(\mu)}
(2\sin\frac{\mu_jg_s}{2})(2\sin\frac{dg_s}{2})^{2g-2}e^{dQ\cdot
\omega}.
\end{align}
These new integer $n_{\mu,g,Q}$ (here $\mu$ denote a partition of a
positive integer) are called the LMOV invariants in this paper.

Although for any toric Calabi-Yau 3-fold with Agangica-Vafa brane
(AV-brane for short) \cite{AV}, we have the method of topological
vertex \cite{AKMV,LLLZ} to compute the open string partition
function and furthermore the open Gromov-Witten invariants
$K_{\mu,g,Q}$, it is difficult to compute the corresponding LMOV
invariants $n_{\mu,g,Q}$ at the righthand side of the formula
(\ref{openmultiplecoveringintro}).

We will study these LMOV invariants $n_{\mu,g,Q}$.  In this paper,
we only focus on a special toric Calabi-Yau 3-fold, i.e. the
resolved conifold $\hat{X}$ with one special Lagrangian submanifold
(AV-brane $D_\tau$ in integer framing $\tau$). More general toric
Calabi-Yau 3-fold will be discussed in a separated paper.

According to the large N duality, the open string theory of
$(\hat{X},D_\tau)$ is the large N duality of the Chern-Simons theory
of $(S^3,U_\tau)$, where $U_\tau$ denotes a framed unknot (trivial
knot) with integer framing $\tau$. The large N duality of
Chern-Simons and topological string theory was proposed by Witten
\cite{W3}, and developed further by \cite{GV2,OV,LMV}. Later,
Mari\~no and Vafa \cite{MV} generalized it to the case of the knot
including integer framing. The large N duality of $(\hat{X},D_\tau)$
and $(S^3,U_\tau)$ is expressed in terms of the following identity:
\begin{align} \label{LargeNdualiyofunknot}
Z_{CS}^{(S^3,U_\tau)}(q,a;\mathbf{x})=Z_{str}^{(\hat{X},D_{\tau})}(g_s,a;\mathbf{x}),
\ q=e^{\sqrt{-1}g_s}
\end{align}
where the explicit expressions of the above two partitions in
identity (\ref{LargeNdualiyofunknot}) are given by the formulas
(\ref{partitionfunctionframedunknot}) and
(\ref{partitonfunctionresolvedconifold}) respectively.  The identity
(\ref{LargeNdualiyofunknot}) implies the Mari\~no-Vafa formula
\cite{MV,LLZ,OP}, a very powerful Hodge integral identity, which
implies various important results in intersection theory of moduli
spaces of curves, see \cite{Zhu0} for a review of the applications
of Mari\~no-Vafa formula. The identity (\ref{LargeNdualiyofunknot})
was proved by J. Zhou \cite{Zhou} based on his previous joint works
with C.-C. Liu and K. Liu \cite{LLZ,LLZ2}.

On the other hand side, through mirror symmetry, the partition
function $Z_{str}^{(\hat{X},D_{\tau})}(g_s,a;\mathbf{x})$ can also
be computed by B-model. The mirror geometry information of
$(\hat{X},D_{\tau})$ is encoded in the mirror curve
$\mathcal{C}_{\hat{X}}$. The disc counting invariants of
$(\hat{X},D_{\tau})$ were given by the coefficients of the
superpotential related to the mirror curve \cite{AV,AKV}, and this
fact was proved in \cite{FL}. Furthermore, the open Gromov-Witten
invariants of higher genus with more holes can be computed by the
Eynard-Orantin topological recursions \cite{EO1}. This approach
named the BKMP conjecture, was proposed by Bouchard, Klemm, Mari\~no
and Pasquetti \cite{BKMP}, and was fully proved in \cite{EO2,FLZ}
for any toric Calabi-Yau 3-fold with  AV-brane, so we can also use
the BKMP method to compute the LMOV invariants for
$(\hat{X},D_{\tau})$.

In conclusion, now we have three different approaches to compute the
open string partition function
$Z_{str}^{(\hat{X},D_{\tau})}(g_s,a;\mathbf{x})$ and their LMOV
invariants $n_{\mu,g,Q}(\tau)$ : topological vertex
\cite{AKMV,LLLZ}, Chern-Simons partition function
(\ref{partitionfunctionframedunknot})  and the BKMP method
\cite{BKMP}.

In this paper, we first compute the genus $0$ LMOV invariants by
BKMP method. At first step, we illustrate the computations of the
mirror curve of $(\hat{X},D_{\tau})$ by using the new approach of
\cite{AV2}. It turns out that the mirror curve is given by:
\begin{align} \label{mirrorcurve}
y-1-a^{-\frac{1}{2}}(-1)^\tau xy^{\tau}(ay-1)=0.
\end{align}
By using mirror curve (\ref{mirrorcurve}), we obtain the genus $0$
with one-hole LMOV invariants $n_{m,0,l-\frac{m}{2}}(\tau)$ which is
denoted by $n_{m,l}(\tau)$ for brevity:
\begin{align*}
n_{m,l}(\tau)=\sum_{d|m,d|l}\frac{\mu(d)}{d^2}c_{\frac{m}{d},\frac{l}{d}}(\tau),
\end{align*}
where
\begin{align*}
c_{m,l}(\tau)=-\frac{(-1)^{m\tau+m+l}}{m^2}\binom{m}{l}\binom{m\tau+l-1}{m-1}.
\end{align*}
We prove the integrality of the number $n_{m,l}(\tau)$.
\begin{theorem}
For any $\tau\in \mathbb{Z}$, $m\geq 1, l\geq 0$, we have
$n_{m,l}(\tau) \in \mathbb{Z}$.
\end{theorem}

For LMOV invariants of genus 0 with two holes, we study the Bergmann
kernel expansion in the BKMP construction, and find an explicit
formula for the LMOV invariants
$n_{(m_1,m_2),0,\frac{m_1+m_2}{2}}(\tau)$ which is denoted by
$n_{(m_1,m_2)}(\tau)$ for short,
\begin{align*}
    n_{(m_1,m_2)}(\tau)& =  \frac{1}{m_1+m_2}\sum_{d\mid m_1,d\mid m_2} \mu(d) (-1)^{(m_1+m_2)(\tau+1)/d} \nonumber \\
                     & \cdot
                     \binom{(m_1\tau+m_1)/d-1}{m_1/d}\binom{(m_2\tau+m_2)/d}{m_2/d}.
\end{align*}
Then, we prove that
\begin{theorem}
For $m_1,m_2\geq 1$, and $\tau\in \mathbb{Z}$,
$n_{(m_1,m_2)}(\tau)\in \mathbb{Z}$.
\end{theorem}

For the genus $0$ LMOV invariants with more than two holes, we can
compute the LMOV invariant $n_{\mu,g,Q}(\tau)$ for general $Q$ by
using the BKMP construction. But it is hard to give an explicit
formula for general $Q$, except the case $Q=\frac{|\mu|}{2}$, in
which the Mari\~no-Vafa formula holds.
\begin{align*}
n_{\mu,0,\frac{|\mu|}{2}}(\tau)=(-1)^{l(\mu)}\sum_{d|\mu}\mu(d)d^{l(\mu)-1}K^{\tau}_{\frac{\mu}{d},0,\frac{|\mu|}{2d}}
\end{align*}
where
\begin{align*}
K^\tau_{\mu,0,\frac{|\mu|}{2}}=(-1)^{|\mu|\tau}[\tau(\tau+1)]^{l(\mu)-1}\prod_{i=1}^{l(\mu)}
\binom{\mu_i(\tau+1)-1}{\mu_i-1}\left(\sum_{i=1}^{l(\mu)}\mu_i\right)^{l(\mu)-3}.
\end{align*}
It is clear that $K^\tau_{\mu,0,\frac{|\mu|}{2}}\in \mathbb{Z}$, for
any $\tau\in \mathbb{Z}$ and since $l(\mu)\geq 3$, it immediately
implies that:
\begin{theorem}
For a partition $\mu$ with $l(\mu)\geq 3$ and $\tau \in \mathbb{Z}$,
$n_{\mu,0,\frac{|\mu|}{2}}(\tau)\in \mathbb{Z}$.
\end{theorem}

Next, we study the Chern-Simons partition function
$Z_{CS}^{(S^3,U_\tau)}(q,a;\mathbf{x})$ whose explicit formula is
given in (\ref{partitionfunctionframedunknot}). Following the works
of \cite{MV,LP3}, we formulate the LMOV conjecture for a general
framed knots, see Conjectures \ref{LMOVframedknot} and
\ref{LMOVframedknot2}. The mathematical structures of the LMOV
conjecture for general link were first studied by K. Liu and P. Peng
\cite{LP1}. Then we formulate the higher genus and one hole LMOV
invariants $n_{m,g,Q}(\tau)$ into a unified generating function
$g_{m}(q,a)$. The integrality of the LMOV invariants
$n_{m,g,Q}(\tau)$ is equivalent to the following theorem which will
be proved in Section 4.5.

\begin{theorem} \label{LMOVframedknot3}
Let $g_m(q,a)=\sum_{d|m}\mu(d)\mathcal{Z}_{m/d}(q^d,a^d)$, where
\begin{align*}
\mathcal{Z}_m(q,a)=(-1)^{m\tau}\sum_{|\nu|=m}\frac{1}{\mathfrak{z}_\nu}\frac{\{m\nu\tau\}}{\{m\}\{m\tau\}}\frac{\{\nu\}_a}{\{\nu\}},
\end{align*}
see formula (\ref{quantuminteger}) for the definitions of the above
quantum integers.  For any integer $m\geq 1$ and any $\tau \in
\mathbb{Z}$, there exist integers $n_{m,g,Q}(\tau)$, such that
\begin{align*}
g_m(q,a)=\sum_{g\geq 0}\sum_{Q}n_{m,g,Q}(\tau)z^{2g-2}a^Q\in
z^{-2}\mathbb{Z}[z^2,a^{\pm \frac{1}{2}}],
\end{align*}
where $z=q^{\frac{1}{2}}-q^{-\frac{1}{2}}=\{1\}$.
\end{theorem}

In Section 5, we introduce the definition of the reduced open string
partition function motivated by the work \cite{AV2}. And we compute
the reduced open string partition function $
\tilde{Z}^{(\mathbb{C}^3,D_\tau)}_{str}(g_s,x) $ for the trivial
Calabi-Yau 3-fold $(\mathbb{C}^3,D_\tau)$ (see the formula
(\ref{partitionfunctionC3})). For brevity, we let
$Z_\tau(q,x)=\tilde{Z}^{(\mathbb{C}^3,D_\tau)}_{str}(g_s,x)$, it
turns out that
\begin{align*}
Z_\tau(q,x)=\sum_{n\geq
0}\frac{(-1)^{n(\tau-1)}q^{\frac{n(n-1)}{2}\tau+\frac{n^2}{2}}}{(1-q)(1-q^2)\cdots(1-q^n)}x^n
\end{align*}

By comparing with the expression of the Hilbert-Poincar\'e series
$P_m(q,t)$ (see formula (\ref{HPseries}) ) of the Cohomological Hall
algebra \cite{KS} of the $m$-loop quiver \cite{Re}, we obtain the
following open string GW/DT correspondence:
\begin{theorem} \label{openstringGWDT}
For $\tau\leq -1$ (i.e. $-\tau\geq 1$), we have
\begin{align*}
Z_{\tau}(q,x)=P_{-\tau}(q,(-1)^{\tau-1}xq^{\frac{1}{2}}).
\end{align*}
\end{theorem}

The main property of the Hilbert-Poincar\'e series $P_{m}(q,t)$ is
the following factorization formula:
\begin{theorem}[Conjecture 3.3 \cite{Re} or Theorem 2.3 \cite{KS}]
\label{factorization}
 There exists a product expansion
\begin{align*}
P_{m}(q,(-1)^{m-1}t)=\prod_{n\geq 1}\prod_{k\geq 0}\prod_{l\geq
0}(1-q^{l-k}t^n)^{-(-1)^{(m-1)n}c_{n,k}}
\end{align*}
for nonnegative integers $c_{n,k}$, such that only finitely many
$c_{n,k}$ are nonzero for any fixed $n$.
\end{theorem}
The series $DT_n^{(m)}(q)=\sum_{k\geq 0}c_{n,k}q^k$ is called the
quantum Donaldson-Thomas invariant in \cite{Re}.

Besides, we also formulate the reduced LMOV conjecture (see
conjecture \ref{redLMOV}), which can be viewed as a weak form of the
LMOV conjecture due to the original work of \cite{OV}. In
particular, the reduced LMOV conjecture in the case of
$(\mathbb{C}^3, D_\tau)$ says:
\begin{conjecture} \label{redLMOVforC3intro} There exist
nonnegative integers $N_{m,k}(\tau)$, and only finitely many
$N_{m,k}(\tau)$ are nonzero for any fixed $m\geq 1$. Such that
\begin{align*}
Z_{\tau}(q,x)=\prod_{m\geq 1}\prod_{k\in \mathbb{Z}}\prod_{l\geq
0}\left(1-q^{\frac{k}{2}+l}x^m\right)^{N_{m,k}(\tau)}.
\end{align*}
\end{conjecture}

Theorem \ref{openstringGWDT} and Theorem \ref{factorization} imply
that, for $\tau\leq -1$, the reduced open string partition function
$Z_{\tau}(q,x)$ on $(\mathbb{C}^3,D_{\tau})$ carries the product
factorization:
\begin{align}
Z_{\tau}(q,x)=\prod_{n\geq 1}\prod_{k\geq 0}\prod_{l\geq
0}(1-q^{\frac{n}{2}+l-k}x^n)^{-(-1)^{(\tau-1)n}c_{n,k}}.
\end{align}
It provides the correspondence of the  Ooguri-Vafa invariants (or
weak LMOV invariants) $N_{m,k}(\tau)$ and the Donaldson-Thomas
invariants $c_{n,k}$ for $\tau\leq -1$.

The toric diagram of the trivial Calabi-Yau 3-fold
$(\mathbb{C}^3,D_{\tau})$ is a topological vertex with one framed
leg. The above open string GW/DT correspondence shows that there is
a corresponding quiver with self loops. Now we can ask the following
questions:

{\bf Questions}: What is the DT correspondence for the reduced open
string partition $\tilde{Z}_{str}^{(\hat{X},D_\tau)}$ of the
resolved conifold $(\hat{X},D_\tau)$?  More general, we can ask, for
a toric Calabi-Yau 3-fold with one out AV-brane $(X,D)$,  if there
exists a corresponding quiver with self-loops, such that the reduced
open partition function $(X,D)$ is equal to the Hilbert-Poincar\'e
series of the Cohomological Hall algebra attached to this quiver?
And also for a framed knot $\mathcal{K}_\tau$, if there exits a
corresponding quiver?

 We will study these questions in our further
work.

The rest of this paper is organized as follow: In section 2, we
review the definitions of topological string partition functions,
free energies, and the integrality structures appearing in
topological strings. We introduce the definitions of Gopakumar-Vafa
invariants in closed strings, and LMOV invariants in open strings.
In section 3, we first review the Witten's Chern-Simons theory for 3
manifolds and links, and the large N duality between the
Chern-Simons theory and the topological strings. Then, a basic
example of for the case of framed unknot was illustrated. We
formulate the LMOV conjecture for the framed knot. In section 4, we
study the LMOV invariants for framed unknot in detail. We first
illustrate the computations of the mirror curve of
$(\hat{X},D_{\tau})$ by using the new approach of \cite{AV2}. Then,
we compute the explicit formulas for genus 0 LMOV invariants by
using mirror curve, and prove the integrality of them. Next, we
formulate the higher genus with one hole LMOV invariants into a
unified generating function by using LMOV conjecture for framed
knot. We prove the integrality of these invariants, i.e. this
generating function lies in a certain integral ring. In section 5,
we introduce the definitions of the reduced partition functions and
establish a correspondence of the open string on
$(\mathbb{C}^3,D_\tau)$ and the Cohomological Hall algebra of a
quiver with self-loops. In section 6, the appendix provides a proof
of the integrality of the other BPS invariants obtained in
\cite{GKS} by our method used in this paper.

\textbf{Acknowledgements.} We would like to thank Prof. Kefeng Liu
and Hao Xu for their interests and comments. The second author
appreciates the collaboration with Qingtao Chen, Kefeng Liu and Pan
Peng which motivates the study of this work.

\section{Topological strings}
\subsection{Closed strings and Gromov-Witten invariants}
Topological strings on a Calabi-Yau 3-fold $X$ have two types, the
A-models and the B-models. The mathematical theory for A-model is
Gromov-Witten theory. Let $\overline{\mathcal{M}}_{g,n}(X,Q)$ be the
moduli space of stable maps $(f, \Sigma_g,p_1,..,p_n)$, where $f:
\Sigma_g\rightarrow X$ is a holomorphic map from the nodal curve
$\Sigma_g$ to the K\"{a}hler manifold $X$ with
$f_*([\Sigma_g])=\beta\in H_2(X,\mathbb{Z})$. In general,
$\overline{\mathcal{M}}_{g,n}(X,Q)$ carries a virtual fundamental
class $[\overline{\mathcal{M}}_{g,n}(X,Q)]^{vir}$ in the sense of
\cite{BF,LT}. The virtual dimension
 is given by:
\begin{align*}
\text{vdim} [\overline{\mathcal{M}}_{g,n}(X,Q)]^{vir}=\int_Q
c_1(X)+(\dim X-3)(1-g)+n.
\end{align*}
When $X$ is a Calabi-Yau 3-fold, i.e. $c_1(X)=0$, then vdim$
[\overline{\mathcal{M}}_{g}(X,Q)]^{vir}=0$. The genus $g$, degree
$Q$ Gromov-Witten invariants of $X$ is defined by
\begin{align*}
K^X_{g,Q}=\int_{[\overline{\mathcal{M}}_{g,0}(X,Q)]^{vir}}1
\end{align*}
which is usually denoted by $K_{g,Q}$ for brevity without any
confusions. In the A-model, the genus $g$ closed free energy
$F_{g}^X$ of $X$ is the generating function of Gromov-Witten
invariants $K_{g,Q}$.
\begin{align*}
F_{g}^X=\sum_{Q\neq 0}K_{g,Q}e^{-Q \cdot\omega},
\end{align*}
where $\omega$ is the K\"ahler class for $X$. We define the total
free energy $F^X$ and partition function $Z^X$ as
\begin{align*}
F^X=\sum_{g\geq 0}g_s^{2g-2}F_{g}^X, \ Z^X=\exp(F^X).
\end{align*}
where $g_s$ is the string coupling constant. The mathematical
computations of the free energy $F^{X}$ is mainly by the method of
localizations \cite{Ko,GP}. Especially, when $X$ is a toric
Calabi-Yau 3-fold, we have a more effective approach to obtain the
partition function $Z^X$ by the method of topological vertex
\cite{AKMV,LLLZ}.

 Usually, the Gromov-Witten invariants
$K_{g,Q}$ are rational numbers, from M-theory, Gopakumar and Vafa
\cite{GV1} expressed the total free energy $F^X$ in terms of the
generating function of the new integer number $N_{g,Q}$ obtained by
counting BPS states:
\begin{align*}
F^X=\sum_{g\geq 0}g_s^{2g-2}\sum_{Q\neq 0}K_{g,Q}e^{-Q
\cdot\omega}=\sum_{g\geq 0, d\geq 1}\sum_{Q\neq
0}\frac{1}{d}N_{g,Q}\left(2\sin\frac{dg_s}{2}\right)^{2g-2}e^{-dQ\cdot
\omega}
\end{align*}
The integrality of the Gopakumar-Vafa invariants $N_{g,Q}$ was first
proved by P. Peng in the case of toric Del Pezzo surfaces
\cite{Peng}. The proof for general toric Calabi-Yau threefolds was
given by Konishi in \cite{Konishi}.

\subsection{Open strings}

Let us now consider the open sector of topological A-model of a
Calabi-Yau 3-fold $X$ with a submanifold $\mathcal{D}$ with dim
$H_{1}(\mathcal{D},\mathbb{Z})=L$. The open sector topological
A-model can be described by holomorphic maps $\phi$ from open
Riemann surface of genus $g$ with $l$-holes $\Sigma_{g,l}$ to $X$,
with Dirichlet condition specified by $\mathcal{D}$.  These
holomorphic maps are called open string instantons. More precisely,
an open string instanton is a holomorphic map $\phi:
\Sigma_{g,h}\rightarrow X$ such that $\partial \Sigma_{g,l}=\cup
_{i=1}^{l}\mathcal{C}_i\rightarrow \mathcal{D}\subset X$ where the
boundary $\partial \Sigma_{g,l}$ of $\Sigma_{g,l}$ consists of $l$
connected components $\mathcal{C}_i$ mapped to Lagrangian manifold
$\mathcal{D}$ of $X$. Therefore, the open string instanton $\phi$ is
described by the following two different kinds of data: the first is
the ``bulk part" which is given  by $\phi_*[\Sigma_{g,l}]=Q\in
H_2(X,\mathcal{L})$, and the second is the ``boundary part" which is
given by $\phi_*[\mathcal{C}_i]=w^\alpha_i\gamma_\alpha$, for
$i=1,..l$, where $\gamma_\alpha$,\ $\alpha=1,..,L$ is a basis of
$H_{1}(\mathcal{D},\mathbb{Z})$ and $w^\alpha_i\in \mathbb{Z}$. Let
$\vec{w}=(w^1,..,w^L)$, and where
$w^\alpha=(w^\alpha_1,...,w^\alpha_l)\in \mathbb{Z}^l$, for
$\alpha=1,...,L$. We expect there exist the corresponding open
Gromov-Witten invariants $K_{\vec{w},g,Q}$ determined by the data
 $\vec{w}, Q$ in the genus $g$. Now, the total free energy $F_{str}^{(X,\mathcal{D})}$ is defined as
\begin{align*}
&F_{str}^{(X,\mathcal{D})}=\sum_{g\geq 0}\sum_{l\geq
1}\sum_{\vec{w}}g_s^{2g-2+l}F_{\vec{w},g}(\omega)\frac{1}{l!}\prod_{\alpha=1}^{L}\prod_{i=1}^{l}TrV^{w^\alpha_i}
\end{align*}
\begin{align*}
F_{\vec{w},g}^{(X,\mathcal{D})}=\sum_{Q\neq
0}K_{\vec{w},g,Q}e^{-Q\cdot \omega}
\end{align*}
where $\omega$ is also the K\"{a}hler class of $X$, and $V$ is a
holonomy matrix of gauge group $U(\infty)$ on the source A-brane
\cite{W3}.

Usually,  we write the total free energy $F_{str}^{(X,\mathcal{D})}$
in the form of the summation over all partitions \cite{MV}.
\begin{align*}
F_{str}^{(X,\mathcal{D})}=\sum_{g\geq 0}\sum_{\vec{\mu} \in
(\mathcal{P}^+)^L}\frac{1}{|Aut(\vec{\mu})|}g_{s}^{2g-2+l(\vec{\mu})}
F_{\vec{\mu},g}^{(X,\mathcal{D})}p_{\vec{\mu}}(\vec{\mathbf{x}}),
\end{align*}
where
$p_{\vec{\mu}}(\vec{\mathbf{x}})=\prod_{\alpha=1}^{L}p_{\mu^\alpha}(\mathbf{x}^\alpha)$,
and for a partition $\mu\in \mathcal{P}^{+}$,
$p_{\mu}(\mathbf{x})=\prod_{i=1}^{h}p_{\mu_i}(\mathbf{x})$.
$p_n(\mathbf{x})$ is the power sum symmetric function \cite{Mac}
given by $p_n(\mathbf{x})=x_1^{n}+x_2^{n}+\cdots$. Where
$\mathcal{P}^+$ denotes the set of all the partitions of positive
integers. Moreover, let $\mathcal{P}=\mathcal{P}^+\cup \{0\}$. The
notations $\mathcal{P}^+, \mathcal{P}$ will be used frequently
throughout this paper.

In the following, we only need to consider the case of $L$=1. It is
useful to write the A-model generating function of
$F_{w,g}^{(X,\mathcal{D})}$ in the fixed genus $g$ as follow:
\begin{align*}
F_{(g,l)}^{(X,\mathcal{D})}=\sum_{w\in
(\mathbb{Z}^+)^l}F_{w,g}^{(X,\mathcal{L})}x_1^{w_1}\cdots x_l^{w_l}.
\end{align*}

The central problem in topological string theory is how to calculate
$F_{(g,l)}^{(X,\mathcal{D})}$. In particular, when $X$ is a toric
Calabi-Yau 3-fold, and $\mathcal{D}$ is a special Lagrangian
submanifold named as Aganagic-Vafa A-brane in the sense of
\cite{AV,AKV}. The open string partition function
$Z_{str}^{(X,\mathcal{D})}=\exp(F_{str}^{(X,\mathcal{D})})$ can be
computed by the method of topological vertex \cite{AKMV,LLLZ}.
However, in this case, there exists another more effective method to
compute $F_{(g,l)}^{(X,\mathcal{D})}$ by using the topological
recursion of Eynard and Orantin \cite{EO1}. This approach was first
proposed by M.Mari\~no \cite{Mar}, and developed further by
Bouchard, Klemm, Mari\~no and Pasquetti  \cite{BKMP}, so the
conjectural equivalence of these two different approaches was called
the BKMP conjecture. Finally, the BKMP conjecture was proved in
\cite{EO2,FLZ}.

\subsection{Integrality}
Let $q=e^{\sqrt{-1}g_s}$, $a=e^{-\omega}$, for the open string free
energy $F_{str}^{(X,\mathcal{D})}$,  we define the generating
functions $f_{\lambda}(q,a)$ by the following expansion formula,
\begin{align*}
F_{str}^{(X,\mathcal{D})}=\sum_{d=1}^{\infty}\frac{1}{d}\sum_{\lambda\in
\mathcal{P}^+}f_{\lambda}(q^d,a^d)s_{\lambda}(\mathbf{x}^d),
\end{align*}
where $s_{\lambda}(\mathbf{x})$ is the Schur symmetric functions.

 Just as in the closed string case \cite{GV1},  the open
topological strings compute the partition function of BPS domain
walls in a related superstring theory \cite{OV}. It follows that
$F^{(X,\mathcal{D})}$ also has an integral expansion. This
integrality structure was further refined in \cite{LM1,LM2,LMV}
which showed that $f_{\lambda}(q,a)$ has the following integral
expansion
\begin{align*}
f_{\lambda}(q,a)=\sum_{g=0}^{\infty}\sum_{Q\neq
0}\sum_{|\mu|=|\lambda|}M_{\lambda\mu}(q)
N_{\mu;g,Q}(q^{\frac{1}{2}}-q^{-\frac{1}{2}})^{2g-2}a^Q,
\end{align*}
where  $N_{\mu;g,Q}$ are integers which compute the net number of
BPS domain walls and $M_{\lambda\mu}(q)$ is defined by
\begin{align} \label{Mlambdamu}
M_{\lambda\mu}(q)=\sum_{\mu}\frac{\chi_{\lambda}(C_{\nu})\chi_{\mu}(C_{\nu})}{\frak{z}_{\nu}}\prod_{j=1}^{l(\nu)}(q^{-\nu_{j}/2}-q^{\nu_{j}/2})
\end{align}
For convenience, we usually introduce the new integers
\begin{align*}
n_{\mu,g,Q}=\sum_{\nu}\chi_{\nu}(C_\mu)N_{\nu,g,Q}.
\end{align*}
\begin{definition}
These integers $N_{\mu,g,Q}$ and  $n_{\mu,g,Q}$ are both called LMOV
invariants.
\end{definition}
Therefore,
\begin{align*}
f_{\lambda}(q,a)=\sum_{g\geq 0}\sum_{Q\neq 0}\sum_{\mu \in
\mathcal{P}}\frac{\chi_{\lambda}(C_{\mu})}{\mathfrak{z}_{\mu}}n_{\mu,g,Q}\prod_{j=1}^{l(\mu)}
(q^{-\frac{\mu_j}{2}}-q^{\frac{\mu_j}{2}})(q^{-\frac{1}{2}}-q^{\frac{1}{2}})^{2g-2}a^Q
\end{align*}
By using the orthogonal relation $
\sum_{\lambda}\frac{\chi_{\lambda}(C_{\mu})\chi_{\lambda}(C_{\nu})}{\mathfrak{z}_{\mu}}=\delta_{\mu,\nu},
$ we obtain the following multiple covering formula for open string
illustrated in \cite{MV}:
\begin{align} \label{openmultiplecovering}
&\sum_{g\geq 0}\sum_{Q\neq
0}g_s^{2g-2+l(\mu)}K_{\mu,g,Q}a^Q\\\nonumber &=\sum_{g\geq
0}\sum_{Q\neq
0}\sum_{d|\mu}\frac{(-1)^{l(\mu)+g}}{\prod_{i=1}^{l(\mu)}\mu_i}d^{l(\mu)-1}n_{\mu/d,g,Q}\prod_{j=1}^{l(\mu)}
(2\sin\frac{\mu_jg_s}{2})(2\sin\frac{dg_s}{2})^{2g-2}a^{dQ}.
\end{align}
Hence we have the following integrality structure conjecture which
is called  the Labastida-Mari\~no-Ooguri-Vafa (LMOV) conjecture for
open string.
\begin{conjecture}[LMOV conjecture for open string]
\label{LMOVforopenstring}
Let $F_{\mu}^{(X,\mathcal{D})}$ be the generating function function
defined by
\begin{align*}
F_{str}^{(X,\mathcal{D})}=\sum_{\mu}F_{\mu}^{(X,\mathcal{D})}p_{\mu}(\mathbf{x}),
\end{align*}
then $F_{\mu}^{(X,\mathcal{D})}$ has the integral expansion as in
the righthand side of the formula (\ref{openmultiplecovering}).
\end{conjecture}
There is no general definition for the open Gromov-Witten invariants
$K_{\mu,g,Q}$. However, just as mentioned in the previous
subsection, when $X$ is a toric Calabi-Yau 3-fold, and $\mathcal{D}$
is  the Aganagic-Vafa A-brane \cite{AV}, the open string partition
function $Z_{str}^{(X,\mathcal{D})}$ can be fully computed by using
the method of topological vertex \cite{AKMV,LLLZ}. The open
Gromov-Witten invariants $K_{\mu,g,Q}$ can also be computed by the
topological recursion formula \cite{BKMP}. It is natural to ask how
to prove the Conjecture \ref{LMOVforopenstring} in the case of toric
Calabi-Yau 3-fold? In this paper, we study this conjecture for the
resolved conifold with one AV-brane in integer framing $\tau$. We
first compute some LMOV integers as predicted by the formula
(\ref{openmultiplecovering}), and then prove that they are really
integers.

\subsection{Lower genus cases}
We illustrate some lower genus cases for the above multiple covering
formula (\ref{openmultiplecovering}). By using the expansion $\sin
x=\sum_{k\geq 1}\frac{x^{2k-1}}{(2k-1)!}$, and taking the
coefficients of $g_s^{2g-2+l(\mu)}a^Q$ in formula
(\ref{openmultiplecovering}), we obtain
\begin{align} \label{multipecoveringgenus0}
K_{\mu,0,Q}=\sum_{d|\mu}(-1)^{l(\mu)}d^{l(\mu)-3}n_{\frac{\mu}{d},0,\frac{Q}{d}},
\end{align}
\begin{align*}
K_{\mu,1,Q}=\sum_{d|\mu}(-1)^{l(\mu)+1}\left(d^{l(\mu)-1}
n_{\frac{\mu}{d},1,\frac{Q}{d}}+\left(\frac{\sum_{j=1}^{l(\mu)}\mu_j^2}{24}d^{l(\mu)-3}-\frac{1}{12}d^{l(\mu)-1}\right)
n_{\frac{\mu}{d},0,\frac{Q}{d}}\right)
\end{align*}
\begin{align*}
&K_{\mu,2,Q}=\sum_{d|\mu}(-1)^{l(\mu)}\left(d^{l(\mu)+1}n_{\frac{\mu}{d},2,\frac{Q}{d}}
+\frac{\sum_{j=1}^{l(\mu)}\mu_j^2}{24}d^{l(\mu)-1}n_{\frac{\mu}{d},1,\frac{Q}{d}}\right.\\\nonumber
&\left.+\left(\frac{\sum_{j=1}^{l(\mu)}\mu_j^4}{1920}d^{l(\mu)-3}+\frac{\sum_{i<j}\mu_i^2\mu_j^2}{576}d^{l(\mu)-3}
-\frac{\sum_{j=1}^{l(\mu)}\mu_j^2}{288}d^{l(\mu)-1}+\frac{1}{240}d^{l(\mu)+1}\right)n_{\frac{\mu}{d},0,\frac{Q}{d}}\right)
\end{align*}
for $g=0$,  $g=1$ and $g=2$ respectively. In fact, these formulas
were firstly computed in \cite{MV}.

Therefore
\begin{align}  \label{multipecoveringlhole}
F_{(0,l)}&=\sum_{|\mu|=l}\sum_{Q}K_{\mu,0,Q}a^Qx_1^{\mu_1}\cdots
x_l^{\mu_l}\\\nonumber
&=\sum_{|\mu|=l}\sum_{Q}\sum_{d|\mu}(-1)^{l(\mu)}d^{l(\mu)-3}n_{\frac{\mu}{d},0,\frac{Q}{d}}a^{Q}x_1^{\mu_1}\cdots
x_l^{\mu_l}\\\nonumber & =(-1)^l\sum_{|\mu|=l}\sum_{Q}\sum_{d\geq
1}d^{l-3}n_{\mu,0,Q}a^{Q}x_1^{d\mu_1}\cdots x_l^{d\mu_l}.
\end{align}
In particular
\begin{align} \label{disccoutingformula}
F_{(0,1)}=-\sum_{m\geq 1}\sum_{d\geq
1}\sum_{Q}\frac{n_{m,0,Q}}{d^2}a^{dQ}x^{dm},
\end{align}
%
and for $g=1,l=1$,
\begin{align*}
F_{(1,1)}&=\sum_{m\geq 0}\sum_{Q}K_{(m),1,Q}a^{Q}x^{m}\\\nonumber
&=\sum_{m\geq
0}\sum_{Q}\left(\sum_{d|m}n_{m/d,1,Q/d}+(\frac{m^2}{24}d^{-2}-\frac{1}{12})n_{m/d,0,Q/d}\right)a^Qx^m\\\nonumber
&=\sum_{m\geq 0}\sum_{Q}\sum_{d\geq
1}\frac{1}{d}\left(n_{m,1,Q}+(\frac{m^2}{24}-\frac{1}{12})\right)a^{dQ}x^{dm}.
\end{align*}

\section{Chern-Simons theory and large N duality}

\subsection{Quantum invariants}
In the seminal paper \cite{W1}, E. Witten defined a topological
invariant of a 3-manifold $M$  as a partition function of quantum
Chern-Simons theory. Let $G$ be a compact gauge group which is a Lie
group, and $M$ be an oriented three-dimensional manifold. Let
$\mathcal{A}$ be a $\mathfrak{g}$-valued connection on $M$ where
$\mathfrak{g}$ is the Lie algebra of $G$. The Chern-Simons \cite{CS}
action is given by
\begin{align*}
S(\mathcal{A})=\frac{k}{4\pi}\int_{M}Tr\left(\mathcal{A}\wedge
d\mathcal{A}+\frac{2}{3}\mathcal{A}\wedge\mathcal{A}
\wedge\mathcal{A}\right)
\end{align*}
where $k$ is an integer called the level.

Chern-Simons partition function is defined as the path integral in
quantum field theory
\begin{align*}
Z^G(M;k)=\int e^{i S(A)}D \mathcal{A}
\end{align*}
where the integral is over the space of all $\mathfrak{g}$-valued
connections $\mathcal{A}$ on $M$. Although it is not rigorous,
Witten developed some techniques to calculate such invariants.

If the 3-manifold $M$ contains a link $\mathcal{L}$, we let
$\mathcal{L}$ be an $L$-component link with $\mathcal{L}=\bigsqcup
_{j=1}^L\mathcal{K}_j$. Define
$W_{R_j}(\mathcal{K}_j)=Tr_{R_j}Hol_{\mathcal{K}_j}(\mathcal{A})=Tr_{R_j}\left(P
\exp\oint_{C_j}\mathcal{A} \right)$ to be the trace of holomony
along $\mathcal{K}_j$ taken in representation $R_j$. Then Witten's
invariant of the pair $(M,\mathcal{L})$ is given by
\begin{align*}
Z^{G}(M,\mathcal{L};\{R_j\};k)=\int e^{iS(\mathcal{A})}\prod_{j=1}^L
W_{R_j}(\mathcal{K}_j)D\mathcal{A}.
\end{align*}
We often use the following normalization form
\begin{align*}
P^G_{R}(M,\mathcal{K};k)=\frac{Z^{G}(M,\mathcal{K};R;k)}{Z^{G}(M;k)}.
\end{align*}
When $M=S^3$ and the lie algebra of $G$ is the semisimple lie
algebra, Reshetikhin and Turaev \cite{RT1,RT2} developed a
systematic way to constructed the above invariants by using the
representation theory of quantum groups. Their construction led to
the definition of the colored HOMFLY-PT invariants  \cite{LM2,LZ},
which can be viewed as the large $N$ limit of the quantum
$U_{q}(sl_N)$ invariants. Usually, we use the notation
$W_{\lambda^1,..,\lambda^L}(\mathcal{L};q,a)$ to denote the
(framing-independent) colored HOMFLY-PT invariants for a (oriented)
link $\mathcal{L}=\bigsqcup _{j=1}^L\mathcal{K}_j$, where each
component $\mathcal{K}_j$ is colored by an irreducible
representation $V_{\lambda^j}$ of $U_{q}(sl_N)$. Some basic
structures for $W_{\lambda^1,..,\lambda^L}(\mathcal{L};q,a)$ were
proved in \cite{LP1,LP2,Zhu}. It is difficult to obtain an explicit
formula of a given link for any irreducible representations
$\lambda$. We refer to \cite{LZ} for an explicit formula for torus
links, and a series of works due to Morozov et al (see for example
\cite{MMM}) proposed many conjectural formulas for the twist knots.
However, in this paper, we only need the following explicit formula
for a trivial knot (unknot) $U$ (for example, see formula (4.6)
in\cite{LP1})
\begin{align} \label{unknotformula}
W_{\lambda}(U;q,a)=\sum_{\mu}\frac{\chi_{\lambda}(\mu)}{\mathfrak{z}_{\mu}}\prod_{i=1}^{l(\mu)}
\frac{a^{\frac{\mu_i}{2}}-a^{-\frac{\mu_i}{2}}}{q^{\frac{\mu_i}{2}}-q^{-\frac{\mu_i}{2}}}.
\end{align}

\subsection{Large N duality} \label{sectionlargeN}
In another fundamental work of Witten \cite{W3}, the $SU(N)$
Chern-Simons gauge theory on a
three-manifold $M$ was interpreted as an open topological string theory on $%
T^*M$ with $N$ topological branes wrapping $M$ inside $T^*M$.
Furthermore, Gopakumar and Vafa \cite{GV2} conjectured that the
large $N$ limit of $SU(N)$ Chern-Simons gauge theory on $S^3$ is
equivalent to the closed topological string theory on the resolved
conifold. Furthermore, Ooguri and Vafa \cite{OV} generalized the
above construction to the case of a knot $\mathcal{K}$ in $S^3$.
They introduced the Chern-Simons partition function
$Z_{CS}^{(S^3,\mathcal{K})}$ for $(S^3,\mathcal{K})$ which is the
generating function of the colored HOMFLY-PT invariants in all
irreducible representations.
\begin{align} \label{chernsimonspartition}
Z_{CS}^{(S^3,\mathcal{K})}(q,a,\mathbf{x})=\sum_{\lambda\in
\mathcal{P}}W_{\lambda}(\mathcal{L},q,a)s_{\lambda}(\mathbf{x}).
\end{align}
Ooguri and Vafa conjectured that for any knot $\mathcal{K}$ in
$S^3$, there exists a corresponding Lagrangian submanifold
$\mathcal{D}_{\mathcal{K}}$, such that the Chern-Simons partition
function $Z_{CS}^{(S^3,\mathcal{D})}$
 is equal to the open topological string
partition function $Z_{str}^{(X,\mathcal{D}_\mathcal{K})}$ on
$(X,\mathcal{D}_{\mathcal{K}})$. They have established this duality
in the case of a trivial knot $U$ in $S^3$, and the link case was
further discussed in \cite{LMV}.

 In general,  we first
should find a way to construct the Lagrangian submanifold
$D_\mathcal{L}$ corresponding to the link $\mathcal{L}$ in geometry.
See \cite{LMV,Koshkin,Tau,DSV} for the constructions for some
special links. Furthermore, if we have found the Lagrangian
submanifold, we need to compute the open sting partition function
under this geometry. For the trivial knot in $S^3$, the dual open
string partition function was computed by J. Li and Y. Song
\cite{LS} and S. Katz and C.-C.M. Liu \cite{KL}.

On the other hand side, Aganagic and Vafa \cite{AV} introduced the
special Lagrangian submanifold in toric Calabi-Yau 3-fold which we
call Aganagic-Vafa A-brane (AV-brane) and studied its mirror
geometry, then they computed the counting of the holomorphic disc
end on AV-brane by using the idea of mirror symmetry. Moreover,
Aganagic and Vafa surprisingly found the computation by using mirror
symmetry and the result from Chern-Simons knot invariants \cite{OV}
are matched. Furthermore, in \cite{AKV}, Aganagic, Klemm and Vafa
investigated the integer ambiguity appearing in the disc counting
and discovered that the corresponding ambiguity in Chern-Simons
theory was described by the framing of the knot. They checked that
the two ambiguities match for the case of the unknot, by comparing
the disk amplitudes on both sides.

Then, Mari\~no and Vafa \cite{MV} generalized the large N duality to
the case of knots with arbitrary framing. They studied carefully and
established the large N duality between a framed unknot in $S^3$ and
the open string theory on resolved conifold with AV-brane  by using
the mathematical approach in \cite{KL}.  By comparing the
coefficient of the highest degree of the K\"{a}hler parameter in
this duality, they derived a remarkable Hodge integral identity
which now is called the Mari\~no-Vafa formula. Two mathematical
proofs for the Mari\~no-Vafa formula were given in \cite{LLZ} and
\cite{OP} respectively. We describe this duality in more details.
For a framed knot $\mathcal{K}_\tau$ with framing $\tau\in
\mathbb{Z}$, we define the framed colored HOMFLYPT invariants
$\mathcal{K}_\tau$ as follow,
\begin{align} \label{framedknotformula}
\mathcal{H}_\lambda(\mathcal{K}_\tau,q,a)=(-1)^{|\lambda|\tau}q^{\frac{\kappa_\lambda\tau}{2}}W_{\lambda}(\mathcal{K},q,a),
\end{align}
where
$\kappa_\lambda=\sum_{i=1}^{l(\lambda)}\lambda_i(\lambda_i-2i+1)$.

 The Chern-Simon partition function for
$(S^3,\mathcal{K}_\tau)$ is given by
\begin{align} \label{partitionfunctionframedunknot}
Z_{CS}^{(S^3,\mathcal{K}_\tau)}(q,a;\mathbf{x})=\sum_{\lambda\in
\mathcal{P}}\mathcal{H}_\lambda(\mathcal{K}_\tau,q,a)s_{\lambda}(\mathbf{x}).
\end{align}
We let $\hat{X}:=\mathcal{O}(-1)\oplus\mathcal{O}(-1)\rightarrow
\mathbb{P}^1$ be the resolved conifold, and $D_{\tau}$ be the
corresponding AV-brane. The open string partition function for
$(\hat{X},D_{\tau})$ has the structure
\begin{align} \label{partitonfunctionresolvedconifold}
Z_{str}^{(\hat{X},D_{\tau})}(g_s,a;\mathbf{x})=\exp\left(
-\sum_{g\geq0,\mu}\frac{\sqrt{-1}^{l(\mu)}}{|Aut(\mu)|}g_s^{2g-2+l(\mu)}F_{\mu,g}^{
\tau}(a)p_{\mu}(\mathbf{x})\right)
\end{align}
where $ F^{\tau}_{\mu,g}(a)=\sum_{Q\in
\mathbb{Z}/2}K^{\tau}_{\mu,g,Q}a^Q $ and $K_{\mu,g,Q}^{\tau}$ is the
open Gromov-Witten invariants defined by
\begin{align*}
K_{\mu,g,Q}^{\tau}=\int_{[\mathcal{M}_{g,l(\mu)}(D^2,S^1|2Q,\mu_1,..,\mu_{l})]}e(\mathcal{V}),
\end{align*}
defined in S. Katz and C.-C. Liu \cite{KL}.  In particular, when
$Q=\frac{|\mu|}{2}$, the
 computations in \cite{KL} gives
\begin{align}\label{MVGW}
&K^{\tau}_{\mu,g,\frac{|\mu|}{2}}=(-1)^{|\mu|\tau}(\tau(\tau+1))^{l(\mu)-1}\\\nonumber
&\prod_{i=1}^{l(\mu)}
\frac{\prod_{j=1}^{\mu_i-1}(\mu_i\tau+j)}{(\mu_i-1)!}\int_{\overline{\mathcal{M}}_{g,l(\mu)}}
\frac{\Lambda_{g}^{\vee}(1)\Lambda_{g}^{\vee}(-\tau-1)\Lambda_g^{\vee}(\tau)}{\prod_{i=1}^{l(\mu)}(1-\mu_j\psi_j)}
\end{align}
where
$\Lambda_g^{\vee}(\tau)=\tau^g-\lambda_1\tau^{g-1}+\cdots+(-1)^g\lambda_g$.
Therefore, the large N duality in this case is given the following
identity:
\begin{align} \label{LargeNdualiyofunknot2}
Z_{CS}^{(S^3,U_\tau)}(q,a;\mathbf{x})=Z_{str}^{(\hat{X},D_{\tau})}(g_s,a;\mathbf{x})
\end{align}
where $q=e^{ig_s}$. By taking the coefficients of
$a^{\frac{|\mu|}{2}}$  of the following equality:
\begin{align*}
[p_{\mu}(\mathbf{x})g_s^{2g-2+l(\mu)}]\log
Z_{CS}^{(S^3,U_\tau)}(q,a;\mathbf{x})=[p_\mu(\mathbf{x})g_s^{2g-2+l(\mu)}]\log
Z_{str}^{(\hat{X},D_{\tau})}(g_s,a;\mathbf{x}),
\end{align*}
 we get the
Mari\~no-Vafa formula which is a Hodge integral identity with triple
$\lambda$ classes. It provides a very powerful tool in studying the
intersection theory of moduli space of curves. From it, we can
derive the Witten conjecture \cite{W2,Ko0}, the ELSV formula
\cite{ELSV}, and various Hodge integral identities, see
\cite{LLZ1,Li,Zhu0}.

Combining the duality ideas above,  together with several new
technical ingredients,  Aganagic, Klemm, Mari\~no and Vafa finally
developed a systematic method, gluing the topological vertex, to
compute all loop topological string amplitudes on toric Calabi-Yau
manifolds \cite{AMV,AKMV}. The mathematical theory for topological
vertex was finally established in \cite{LLLZ}.  This method give an
effective way to compute both the closed and open string partition
function for a toric Calabi-Yau 3-fold  with AV-brane. Therefore, we
have an explicit formula for the partition function of resolved
conifold $Z_{str}^{(\hat{X},D_\tau)}(g_s,a)$, by comparting the
explicit formula $Z_{CS}^{(S^3,U_\tau)}(q,a)$ of Chern-Simons
partition function describe above,  J. Zhou proved the identity
(\ref{LargeNdualiyofunknot2}) in \cite{Zhou} based on the results of
their previous works \cite{LLZ,LLZ2,LLLZ}.

\subsection{Integrality of the quantum invariants}
Now, let us collect the above discussions together. Let
$\mathcal{L}$ be a link in $S^3$, the large N duality predicts there
exists a Lagrangian submanifold $\mathcal{D}_\mathcal{L}$ in the
resolved confold $\hat{X}$, and provides us the identity
(\ref{LargeNdualiyofunknot2}). Since
$Z_{str}^{(\hat{X},\mathcal{D}_\mathcal{L})}(g_s,a,\mathbf{x})$ has
the integrality structures by the discussions in section 2.3, it
implies that $Z_{CS}^{(S^3,\mathcal{L})}(q,a,\mathbf{x})$ also
inherits the integrality structure. Usually, this integrality
structure is called the LMOV conjecture for link in \cite{LP1}.
Furthermore, as mentioned previously, the large N duality was
generalized to the case of framed knot $\mathcal{K}_\tau$ with
framing $\tau\in \mathbb{Z}$ in \cite{MV}, with the Chern-Simons
partition $Z_{CS}^{(S^3,\mathcal{K}_\tau)}$ for framed knot
$\mathcal{K}_\tau$ given in formula
(\ref{partitionfunctionframedunknot}). For convenience, we only
formulate the LMOV conjecture for framed knot $\mathcal{K}_\tau$ in
the following, although the conjecture should also holds for any
framed link, see \cite{LP3}.
\begin{conjecture}[LMOV conjecture for framed knot or framed LMOV conjecture]
\label{LMOVframedknot} Let
\begin{align*}
F_{CS}^{(S^3,\mathcal{K}_\tau)}(q,a,\mathbf{x})=\log
Z_{CS}^{(S^3,\mathcal{K}_\tau)}(q,a,\mathbf{x})
\end{align*}
be the Chern-Simons free energy for a framed knot $\mathcal{K}_\tau$
in $S^3$. Then there exist functions
$f_{\lambda}(\mathcal{K}_\tau;q,a)$ such that
\begin{align*}
F_{CS}^{(S^3,\mathcal{K}_\tau)}(q,a,\mathbf{x})=\sum_{d=1}^\infty\frac{1}{d}\sum_{\lambda\in
\mathcal{P},\lambda\neq
0}f_{\lambda}(\mathcal{K}_\tau;q^d,a^d)s_{\lambda}(\mathbf{x}^d).
\end{align*}
Let $
\hat{f}_{\mu}(\mathcal{K}_\tau;q,a)=\sum_{\lambda}f_{\lambda}(\mathcal{K}_\tau;q,a)M_{\lambda\mu}(q)^{-1},
$ where $M_{\lambda\mu}(q)$ is defined in the formula
(\ref{Mlambdamu}). Denote $z=q^{\frac{1}{2}}-q^{-\frac{1}{2}}$, then
for any $\mu\in \mathcal{P}^+$, there are integers
$N_{\mu,g,Q}(\tau)$ such that
\begin{align*}
\hat{f}_{\mu}(\mathcal{K}_\tau;q,a)=\sum_{g\geq
0}\sum_{Q}N_{\mu,g,Q}(\tau)z^{2g-2}a^Q\in
z^{-2}\mathbb{Z}[z^{2},a^{\pm \frac{1}{2}}].
\end{align*}
Therefore,
\begin{align*}
\mathfrak{z}_\mu\hat{g}_{\mu}(\mathcal{K}_\tau;q,a)&=\sum_{\nu}\chi_{\nu}(C_\mu)\hat{f}_\nu(\mathcal{K}_\tau;q,a)\\\nonumber
&=\sum_{g\geq 0}\sum_{Q}n_{\mu,g,Q}(\tau)z^{2g-2}a^Q\in
z^{-2}\mathbb{Z}[z^{2},a^{\pm \frac{1}{2}}].
\end{align*}
where
$n_{\mu,g,Q}(\tau)=\sum_{\nu}\chi_{\nu}(C_\mu)N_{\nu,g,Q}(\tau)$.
\end{conjecture}
 K. Liu and P. Peng \cite{LP1} first studied the
mathematical structures of LMOV conjecture for general links (as to
the Chern-Simons partition (\ref{chernsimonspartition})), which is
equivalent to the framed LMOV conjecture for any links in framing
zero. They provided a proof for this case by using cut-and-join
analysis and the cabling technique \cite{LZ}. Motivated by the work
\cite{MV}, K. Liu and P. Peng \cite{LP3} formulated the framed LMOV
conjecture for any links(as to the Chern-Simons partition function
(\ref{partitionfunctionframedunknot}). In \cite{CLPZ}, the second
author, together with Q. Chen, K. Liu and P. Peng,  developed the
ideas in \cite{LP3} to study the mathematical structures in framed
LMOV and formulate congruence skein relations for colored HOMFLY-PT
invariants.

\section{LMOV invariants for framed unknot $U_\tau$}
In Section \ref{sectionlargeN}, we have showed that, for a framed
unknot $U_\tau$ in $S^3$, the large N duality holds \cite{Zhou}:
\begin{align*}
Z_{CS}^{(S^3,U_\tau)}(q,a;\mathbf{x})=Z_{str}^{(\hat{X},D_{\tau})}(g_s,a;\mathbf{x}),
\  q=e^{\sqrt{-1}g_s}.
\end{align*}
So one can compute LMOV invariants completely by using the colored
HOMFLY-PT invariants of the framed unknot $U_\tau$. On the other
hand side, by using mirror symmetry, one can also compute the
partition function $Z_{str}^{(\hat{X},D_{\tau})}(g_s,a;\mathbf{x})$
from B-model. The mirror geometry information of
$(\hat{X},D_{\tau})$ is encoded in the mirror curve
$\mathcal{C}_{\hat{X}}$. The disc counting information of
$(\hat{X},D_{\tau})$ was given by the superpotential related  to the
mirror curve \cite{AV,AKV}, and this fact was proved in \cite{FL}.

Furthermore, the open Gromov-Witten invariants with higher genus
with more holes can be computed by the Eynard-Orantin topological
recursions \cite{EO1}. This approach named the BKMP conjecture, was
proposed by Bouchard, Klemm, Mari\~no and Pasquetti \cite{BKMP}, and
was fully proved in \cite{EO2,FLZ} for any toric Calabi-Yau 3-fold
with  AV-brane, so we can also use the BKMP method to compute the
LMOV invariants for $(\hat{X},D_{\tau})$. To determine the mirror
curve of $(\hat{X},D_{\tau})$, there is standard method in toric
geometry. However, in \cite{AV2}, Aganagic and Vafa proposed another
effective method to compute the mirror curve, their method can be
applied to the more general large N geometry of any knot in $S^3$.
The rest contents of this section will be organized as follow, we
first illustrate the computations of the mirror curve of
$(\hat{X},D_{\tau})$ by using the new approach of \cite{AV}. Then,
we compute the explicit formulas for genus 0 LMOV invariants, and
prove the integrality of them. Next, we formulate the higher genus
with one hole LMOV invariants into a unified generating function,
and we prove this generating function lies in a certain integral
ring.

\subsection{a-deformed A-polynomial as the mirror curve}
The method used in \cite{AV2} to compute the mirror curve is based
on the fact that, colored HOMFLY-PT invariants colored by a
partition with a single row is a $q$-holonomic function, this fact
was conjectured and used in many literatures, such as
\cite{FGSA,FGS}, and was finally proved in \cite{GLL}. In fact, such
idea can go back to \cite{GL}.

Now, we illustrate the computation for the framed unknot $U_\tau$.
We first compute the noncommutative a-deformed $A$-polynomial(it is
called the Q-deformed A-polynomial in \cite{AV2}, the variable $Q$
in\cite{AV2} is the variable $a$ here) for $U_\tau$.

By formula (\ref{unknotformula}) ,The colored HOMFLY-PT invariants
colored by partition $(n)$ for the unknot $U$ is given by
\begin{align*}
W_{n}(U;q,a)=\frac{a^{\frac{1}{2}}-a^{-\frac{1}{2}}}{q^{\frac{1}{2}}-q^{-\frac{1}{2}}}\cdots
\frac{a^{\frac{1}{2}}q^{\frac{n-1}{2}}-a^{-\frac{1}{2}}q^{\frac{-n-1}{2}}}{q^{\frac{n}{2}}-q^{-\frac{n}{2}}}
\end{align*}
It gives the recursion
\begin{align*}
(q^{n+1}-1)W_{n+1}(U;q,a)-(a^{\frac{1}{2}}q^{n+\frac{1}{2}}-a^{-\frac{1}{2}}q^{\frac{1}{2}})W_n(U;q,a)=0.
\end{align*}
By formula (\ref{framedknotformula}), the framed colored HOMFLY-PT
invariants for the framed unknot with framing $\tau\in \mathbb{Z}$
is
\begin{align*}
\mathcal{H}_n(U_\tau;q,a)=(-1)^{n\tau}q^{\frac{n(n-1)}{2}\tau}W_n(U;q,a).
\end{align*}
So we get the recursion for $\mathcal{H}_n(U_\tau;q,a)$ as follow
\begin{align} \label{frameunknotrecursion}
(-1)^\tau(q^{n+1}-1)\mathcal{H}_{n+1}(U_\tau;q,a)-(a^{\frac{1}{2}}q^{n+\frac{1}{2}}-a^{-\frac{1}{2}}q^{\frac{1}{2}})q^{n\tau}\mathcal{H}_n(U_\tau;q,a)=0.
\end{align}

For a general series $\{\mathcal{H}_n(q,a)\}_{n\geq 0}$,  we
introduce two operators $M$ and $L$ act on
$\{\mathcal{H}_n(q,a)\}_{n\geq 0}$ as follow:
\begin{align*}
M\mathcal{H}_n=q^{n}\mathcal{H}_n,\
L\mathcal{H}_{n}=\mathcal{H}_{n+1}.
\end{align*}
then $LM=qML$.
\begin{definition}
The noncommutative a-deformed A-polynomial for series
$\{\mathcal{H}_n(q,a)\}_{n\geq 0}$ is a polynomial
$\hat{A}(M,L;q,a)$ of operators $M, L$, such that
\begin{align*}
\hat{A}(M,L;q,a)\mathcal{H}_n(q,a)=0, \text{for } \ n\geq 0.
\end{align*}
and $A(M,L;a)=\lim_{q\rightarrow 1}\hat{A}(M,L;q,a)$ is called the
a-deformed A-polynomial.
\end{definition}
Therefore, from the recursion  (\ref{frameunknotrecursion}),  we
obtain the noncommutative a-deformed A-polynomial for $U_\tau$ as
follow:
\begin{align*}
\hat{A}_{U_\tau}(M,L,q;a)=(-1)^\tau(qM-1)L-M^{\tau}(a^{\frac{1}{2}}q^{\frac{1}{2}}M-a^{-\frac{1}{2}}q^{\frac{1}{2}}).
\end{align*}
and the a-deformed A-polynomial is
\begin{align*}
A_{U_\tau}(M,L;a)=\lim_{q\rightarrow
1}\hat{A}(M,L,q;a)=(-1)^\tau(M-1)L-M^{\tau}(a^{\frac{1}{2}}M-a^{-\frac{1}{2}}),
\end{align*}

In order to get the mirror curve of $U_\tau$, we need the following
general result which is written in the following lemma. Let
$Z(x)=\sum_{k\geq 0}\mathcal{H}_k(q,a)x^k$ be a generating function
of the series $\{\mathcal{H}_k(q,a)|k\geq 0\}$. We also introduce
two operators $\hat{x},\hat{y}$ act on $Z(x)$ as follow:
\begin{align*}
\hat{x}Z(x)=xZ(x), \ \hat{y}Z(x)=Z(qx).
\end{align*}
then $\hat{y}\hat{x}=q\hat{x}\hat{y}$. It is easy to obtain the
following result (see Lemma 2.1 in \cite{GKS} for the similar
statement).
\begin{proposition}
Given a noncommutative A-polynomial
$\hat{A}(M,L,q,a)=\sum_{i,j}c_{i,j}M^iL^j$ for the series
$\{\mathcal{H}_k(q,a)|k\geq 0\}$, then we have
\begin{align} \label{ApolynomialforZ}
\hat{A}(\hat{y},\hat{x}^{-1},q,a)Z(x)=\sum_{i,j}\sum_{-j\leq k\leq
-1}\mathcal{H}_{k+j}q^{ki}x^k.
\end{align}
\end{proposition}
\begin{proof}
Since
\begin{align*}
\hat{A}(\hat{y},\hat{x}^{-1},q,a)Z(x)&=\sum_{i,j}c_{i,j}\hat{y}^i\hat{x}^{-j}Z(x)
\\\nonumber
&=\sum_{i,j}c_{i,j}q^{-ij}x^{-j}Z(q^ix)\\\nonumber
&=\sum_{i,j}c_{i,j}\sum_{n\geq
0}\mathcal{H}_nq^{(n-j)i}x^{n-j}\\\nonumber
&=\sum_{i,j}c_{i,j}\sum_{k\geq 0}\mathcal{H}_{k+j}q^{k
i}x^k+\sum_{i,j}\sum_{-j\leq k\leq -1}a_{k+j}q^{ki}x^k.
\end{align*}
and by the definitions of the operators $M,L$,
$\hat{A}(M,L,q,a)\mathcal{H}_k=0$ gives
\begin{align*}
\sum_{i,j}c_{i,j}q^{ki}\mathcal{H}_{k+j}=0, \text{for} \ k\geq 0.
\end{align*}
We obtain the formula (\ref{ApolynomialforZ}).
\end{proof}
Finally, the mirror curve is given by
\begin{align*}
A(y,x^{-1};a)=\lim_{q\rightarrow
1}\hat{A}(\hat{y},\hat{x}^{-1};q,a)=0
\end{align*}
In our case, the mirror curve is:
\begin{align} \label{mirrorcurve2}
A_{U_\tau}(y,x^{-1};a)=y-1-a^{-\frac{1}{2}}(-1)^\tau
xy^{\tau}(ay-1)=0.
\end{align}

\subsection{Disc countings}
For convenience, we let $X=a^{-\frac{1}{2}}(-1)^\tau x$, and
$Y=1-y$, then the mirror curve (\ref{mirrorcurve2}) becomes the
functional equation
\begin{align}  \label{mirrorcurvesimple}
Y=X(1-Y)^\tau(1-a(1-Y)).
\end{align}
In order to solve the above equation, we introduce the following
Lagrangian inversion formula \cite{Stanley}.
\begin{lemma}
Let $\phi(\lambda)$ be an invertible formal power series in the
indeterminate $\lambda$. Then the functional equation $Y=X\phi(Y)$
has a unique formal power series solution $Y=Y(X)$. Moreover, if $f$
is a formal power series, then
\begin{align} \label{lagrangeinversion}
f(Y(X))=f(0)+\sum_{n\geq
1}\frac{X^n}{n}\left[\frac{df(\lambda)}{d\lambda}\phi(\lambda)^{n}\right]_{\lambda^{n-1}}
\end{align}
\end{lemma}
\begin{remark}
In the following, we will frequently use the binomial coefficient
$\binom{n}{k}$ for all $n\in \mathbb{Z}$. That means for $n<0$, we
define $\binom{n}{k}=(-1)^k\binom{-n+k-1}{k}$.
\end{remark}

In our case, we take $ \phi(Y)=(1-Y)^\tau(1-a(1-Y)). $ Let
$f(Y)=1-Y$, by formula (\ref{lagrangeinversion}), we obtain
\begin{align*}
y(X)=1-Y(X)=1+\sum_{n\geq 1}\frac{X^n}{n}\sum_{i\geq
0}(-1)^{n+i}\binom{n}{i}\binom{n\tau +i}{n-1}a^i
\end{align*}
since $\phi(\lambda)^n$ has the expansion
\begin{align*}
\phi(\lambda)^n&=(1-\lambda)^{n\tau}(1-a(1-\lambda))^n\\\nonumber
&=\sum_{i\geq 0}\binom{n}{i}(-a)^i(1-\lambda)^{n\tau+i} \\\nonumber
&=\sum_{i,j\geq
0}\binom{n}{i}(-1)^{i+j}\binom{n\tau+i}{j}a^i\lambda^j.
\end{align*}
Moreover, if we let $f(Y(X))=\log(1-Y(X))$, then
\begin{align*}
\left[\frac{df(\lambda)}{d\lambda}\phi(\lambda)^{n}\right]_{\lambda^{n-1}}&=\sum_{i\geq
0}(-1)^{i}\binom{n}{i}\sum_{j=0}^{n-1}(-1)^{j+1}\binom{n\tau+i}{j}a^{i}\\\nonumber
&=\sum_{i\geq
0}(-1)^{i}\binom{n}{i}(-1)^{n}\binom{n\tau+i-1}{n-1}a^{i}
\end{align*}
where we have used the combinatoric identity:
\begin{align*}
\sum_{j=0}^{n-1}(-1)^{j+1}\binom{m}{j}=(-1)^n\binom{m-1}{n-1}.
\end{align*}
Formula (\ref{lagrangeinversion}) gives
\begin{align*}
\log(y(X))=\log(1-Y(X))=\sum_{n\geq 1}\frac{X^n}{n}\sum_{i\geq
0}(-1)^{n+i}\binom{n}{i}\binom{n\tau+i-1}{n-1}a^{i}.
\end{align*}
i.e.
\begin{align*}
\log(y(x))=\sum_{n\geq 1}\frac{x^n}{n}\sum_{i\geq
0}(-1)^{n\tau+n+i}\binom{n}{i}\binom{n\tau+i-1}{n-1}a^{i-\frac{n}{2}}.
\end{align*}
By BKMP's construction in genus 0 with one hole, one obtains
\begin{align} \label{disccoutingformulaunknot}
F_{(0,1)}&=\int \log(y(x))\frac{dx}{x}\\\nonumber &=\sum_{n\geq
1}\frac{x^n}{n^2}\sum_{i\geq
0}(-1)^{n\tau+n+i}\binom{n}{i}\binom{n\tau+i-1}{n-1}a^{i-\frac{n}{2}}.
\end{align}
By formula (\ref{disccoutingformula}), and if we let
$n_{m,l}(\tau)=n_{m,0,l-\frac{m}{2}}(\tau)$, then
\begin{align} \label{disccoutingformula2}
F_{(0,1)}=-\sum_{m\geq
1}\sum_{d|m,d|l}d^{-2}n_{\frac{m}{d},\frac{l}{d}}(\tau)x^ma^{l-\frac{m}{2}}.
\end{align}
If we let
\begin{align*}
c_{m,l}(\tau)=-\frac{(-1)^{m\tau+m+l}}{m^2}\binom{m}{l}\binom{m\tau+l-1}{m-1},
\end{align*}
by comparing the coefficients of $x^ma^{l-\frac{m}{2}}$ in
(\ref{disccoutingformula2}) and (\ref{disccoutingformulaunknot}), we
have
\begin{align*}
c_{m,l}(\tau)=\sum_{d|m,d|l}\frac{n_{m/d,l/d}(\tau)}{d^2}.
\end{align*}
By M\"{o}bius inversion formula,
\begin{align*}
n_{m,l}(\tau)=\sum_{d|m,d|l}\frac{\mu(d)}{d^2}c_{\frac{m}{d},\frac{l}{d}}(\tau).
\end{align*}

\begin{theorem} \label{integralitythmonehole}
For any $\tau\in \mathbb{Z}$, $m\geq 1, l\geq 0$, we have
$n_{m,l}(\tau) \in \mathbb{Z}$.
\end{theorem}
Before proving Theorem \ref{integralitythmonehole}, we define the
following function, for nonnegative integer $n$ and prime number
$p$,
\begin{equation*}
    f_p(n)=\prod_{i=1,p\nmid i}^n i = \frac{n!}{p^{[n/p]}[n/p]!}.
\end{equation*}

\begin{lemma} \label{functionfp}
    For odd prime numbers $p$ and $\alpha\geq 1$ or for $p=2$, $\alpha\geq 2$,
    we have $p^{2\alpha}\mid f_p(p^{\alpha} n)-f_p(p^{\alpha})^n$. For $p=2, \alpha=1$, $f_2(2n)\equiv (-1)^{[n/2]}\pmod{4}$ \label{luo:lemma1-0}
\end{lemma}

\begin{proof}
 With $\alpha\geq 2$ or $p>2$, $p^{\alpha-1}(p-1)$ is even,
 \begin{align*}
     &f_p(p^\alpha n)-f_p(p^\alpha(n-1))f_p(p^\alpha) \\
     &= f_p(p^\alpha(n-1))\left(\prod_{i=1,p\nmid i}^{p^\alpha} (p^{\alpha}(n-1)+i)-f_p(p^\alpha)\right) \\
     & \equiv p^{\alpha}(n-1) f_p(p^\alpha(n-1))f_p(p^\alpha)\left(\sum_{i=1,p\nmid i}^{p^\alpha}\frac{1}{i}\right) \pmod{p^{2\alpha}} \\
     & \equiv  p^{\alpha}(n-1) f_p(p^\alpha(n-1))f_p(p^\alpha)\left(\sum_{i=1,p\nmid i}^{[p^\alpha/2]} (\frac{1}{i}+\frac{1}{p^{\alpha}-i})\right) \pmod{p^{2\alpha}} \\
     & \equiv  p^{\alpha}(n-1) f_p(p^\alpha(n-1))f_p(p^\alpha)\left(\sum_{i=1,p\nmid i}^{[p^\alpha/2]} \frac{p^{\alpha}}{i(p^\alpha-i)} \right)\equiv 0, \pmod{p^{2\alpha}}
\end{align*}

Thus the first part of the Lemma is proved by induction. For $p=2,
\alpha=1$, the formula is straightforward.
\end{proof}

\begin{lemma}
    For odd prime number $p$ and $m=p^\alpha a, l=p^\beta b$, $p\nmid a, p\nmid b$, $\alpha\geq 1, \beta \geq 0$, we have
    \begin{equation}
        p^{2\alpha} \mid \binom{m}{l}\binom{m\tau+l-1}{m-1}-\binom{\frac{m}{p}}{\frac{l}{p}}\binom{\frac{m\tau+l}{p}-1}{\frac{m}{p}-1}
    \end{equation}
    where for $\beta=0$, the second term is defined to be zero. \label{luo:lemma2-0}
\end{lemma}
\begin{proof}
    \begin{align}
        & \binom{m}{l}\binom{m\tau+l-1}{m-1}-\binom{\frac{m}{p}}{\frac{l}{p}}\binom{\frac{m\tau+l}{p}-1}{\frac{m}{p}-1} \nonumber \\
        & = \binom{\frac{m}{p}}{\frac{l}{p}}\binom{\frac{m\tau+l}{p}-1}{\frac{m}{p}-1} \left( \frac{f_p(m)}{f_p(l)f_p(m-l)}\cdot \frac{f_p(m\tau+l)}{f_p(m)f_p(m(\tau-1)+l)}-1\right) \label{luo:eq5-0}
    \end{align}

    Write $\binom{\frac{m}{p}}{\frac{l}{p}}=\frac{m}{l}\binom{\frac{m}{p}-1}{\frac{l}{p}-1}$ and $\binom{\frac{m\tau+l}{p}-1}{\frac{m}{p}-1}=\frac{m}{m\tau +l} \binom{\frac{m\tau+l}{p}}{\frac{m}{p}}$, both are divisible by $p^{\max(\alpha-\beta,0)}$.
    Each element of $\{m, l, m-l, m\tau+l, m(\tau-1)+l\}$ is divisible by $p^{\min(\alpha,\beta)}$, so by Lemma~\ref{luo:lemma1-0},
\begin{equation}
    \frac{f_p(m)}{f_p(l)f_p(m-l)}\cdot \frac{f_p(m\tau+l)}{f_p(m)f_p(m(\tau-1)+l)}-1 \label{luo:term1-0}
\end{equation}
is divisible by $p^{2\min(\alpha,\beta)}$ (including the case
$\beta=0$) in $p$-adic number field. Thus (\ref{luo:eq5-0}) is
divisible by
$p^{2\max(\alpha-\beta,0)+2\min(\alpha,\beta)}=p^{2\alpha}$.
\end{proof}

\begin{lemma}
    For $m=2^\alpha a, l=2^\beta b, \alpha\geq 1, \beta \geq 0$,
    \[ 2^{2\alpha}\mid (-1)^{m\tau+m+l}\binom{m}{l}\binom{m\tau+l-1}{m-1}-(-1)^{\frac{m\tau+m+l}{2}}\binom{\frac{m}{2}}{\frac{l}{2}}\binom{\frac{m\tau+l}{2}-1}{\frac{m}{2}-1},\]
    where the second term is set to zero for $\beta=0$. \label{luo:lemma3-0}
\end{lemma}

\begin{proof}
    For the case $\alpha\geq 2, \beta\geq 2$, both $m\tau+m+l$ and $(m\tau+m+l)/2$ are even, the Lemma is proved as in Lemma~\ref{luo:lemma2-0}. For the case $\beta=0$, both $\binom{m}{l}$ and $\binom{m\tau+l-1}{m-1}$ are divisible by $2^\alpha$, and the Lemma is also proved.
    For remaining cases $\alpha>\beta =1$ or $\beta\geq \alpha=1$, we compute similarly as (\ref{luo:eq5-0}),
    \begin{align}
        & (-1)^{m\tau+m+l}\binom{m}{l}\binom{m\tau+l-1}{m-1}-(-1)^{\frac{m\tau+m+l}{2}}\binom{\frac{m}{2}}{\frac{l}{2}}\binom{\frac{m\tau+l}{2}-1}{\frac{m}{2}-1}  \nonumber \\
        & = \binom{\frac{m}{2}}{\frac{l}{2}}\binom{\frac{m\tau+l}{2}-1}{\frac{m}{2}-1} \left( \frac{f_2(m)}{f_2(l)f_2(m-l)}\cdot \frac{f_2(m\tau+l)}{f_2(m(\tau-1)+l)f_2(m)}-(-1)^{\frac{m\tau+m+l}{2}}\right) \label{luo:eq6-0}
    \end{align}
    Both $\binom{\frac{m}{2}}{\frac{l}{2}}$ and $\binom{\frac{m\tau+l}{2}-1}{\frac{m}{2}-1}$ are divisible by $2^{\alpha-1}$, it suffice to prove that the third factor is divisible by $4$, which is, by Lemma~\ref{luo:lemma1-0},
    \begin{equation}
        (-1)^{[\frac{l}{4}]+[\frac{m-l}{4}]+[\frac{m\tau+l}{4}]+[\frac{m(\tau-1)+l}{4}]}-(-1)^{\frac{m\tau+m+l}{2}}. \pmod{4} \nonumber
    \end{equation}
    It is divisible by $4$ if
    \begin{equation}
        [\frac{l}{4}]+[\frac{m-l}{4}]+[\frac{m\tau+l}{4}]+[\frac{m(\tau-1)+l}{4}]+\frac{m\tau+m+l}{2}\label{luo:eq7-0}
    \end{equation}
    is even.
Parity of (\ref{luo:eq7-0})  depends only on $\tau\pmod{2}$. For
$\tau=1$, (\ref{luo:eq7-0}) reduces to
$[l/4]+[(m-l)/4]+[(m+l)/4]+[l/4]+l/2$. For $\tau=0$, it reduces to
$[l/4]+[(m-l)/4]+[l/4]+[(l-m)/4]+(m+l)/2$. Both are obviously even.

\end{proof}
Now, we can finish the proof of Theorem \ref{integralitythmonehole}.

\begin{proof}
    For a prime number $p\mid m$, write $m=p^\alpha a, p\nmid a$.
    \begin{align}
        n_{m,l}(\tau)&=\sum_{d\mid m, d\mid l} \frac{\mu(d)}{d^2} c_{\frac{m}{d},\frac{l}{d}}(\tau)  \nonumber \\
                     &= \frac{1}{m^2} \sum_{d\mid m, d\mid l}\mu(d)(-1)^{\frac{m\tau+m+l}{d}} \binom{\frac{m}{d}}{\frac{l}{d}}\binom{\frac{m\tau+l}{d}-1}{\frac{m}{d}-1} \nonumber \\
        &=\frac{1}{m^2} \sum_{d\mid m, d\mid l, p\nmid d} \mu(d) \left((-1)^{\frac{m\tau+m+l}{d}}\binom{\frac{m}{d}}{\frac{l}{d}}\binom{\frac{m\tau+l}{d}-1}{\frac{m}{d}-1}-(-1)^{\frac{m\tau+m+l}{dp}} \binom{\frac{m}{dp}}{\frac{l}{dp}}\binom{\frac{m\tau+l}{dp}-1}{\frac{m}{dp}-1}\right) \label{luo:eq9-0}
    \end{align}
    where for $pd\nmid l$, second term of (\ref{luo:eq9-0}) is understood to be zero. For odd prime number $p$, $\frac{m\tau+m+l}{d}$ and $\frac{m\tau+m+l}{dp}$ have the same parity. Since $p^\alpha\mid\mid\frac{m}{d}$, $p^{2\alpha}$ divides the summand in (\ref{luo:eq9-0})  by Lemma~\ref{luo:lemma2-0}. For $p=2$, it is divisible by $2^{2\alpha}$ by  Lemma~\ref{luo:lemma3-0}.

    $p^{2\alpha}$ divides the sum in (\ref{luo:eq9-0}) for every $p^\alpha \mid\mid m$, thus $n_{m,l}$ is an integer.
\end{proof}

\subsection{Annulus counting}
The  Bergmann kernel of the curve (\ref{mirrorcurvesimple}) is
\begin{align*}
B(X_1,X_2)=\frac{dY_1dY_2}{(Y_1-Y_2)^2}.
\end{align*}
By the construction of BKMP \cite{BKMP}, the annulus amplitude is
calculated by the integral
\begin{align*}
\int
\left(B(X_1,X_2)-\frac{dX_1dX_2}{(X_1-X_2)^2}\right)=\ln\left(\frac{Y_2(X_2)-Y_1(X_1)}{X_2-X_1}\right)
\end{align*}

More precisely, for $m_1,m_2\geq 1$, the coefficients
$\left[\ln\left(\frac{Y_2(X_2)-Y_1(X_1)}{X_2-X_1}\right)\right]_{x_1^{m_1}x_2^{m_2}a^l}$
gives the annulus Gromov-Witten invariants $K_{(m_1,m_2),0,l}$.

Let
$b_{n,i}=\frac{(-1)^{n+i}}{n+1}\binom{n+1}{i}\binom{(n+1)\tau+i}{n}$
and $b_n=\sum_{i\geq 0}b_{n,i}a^i$. In particular $b_0=1-a$. Then
\begin{align*}
Y(X)=\sum_{n\geq 1}b_nX^n.
\end{align*}
and
\begin{align*}
\frac{Y_2(X_2)-Y_1(X_1)}{X_2-X_1}=(1-a)+\sum_{n\geq
1}b_n\left(\sum_{i=0}^{n}X_1^iX_2^{n-i}\right).
\end{align*}
Let $\tilde{b}_{m,l}=\sum_{i=0}^lb_{m,i}$  and
$\tilde{b}_{m}=\sum_{l=0}\tilde{b}_{m,l}a^l$. For $m_1\geq 1,
m_2\geq 1$, the coefficients $c_{(m_1,m_2)}$ of
$[X_1^{m_1}X_2^{m_2}]$ in the expansion
\begin{align*}
\ln\left(1+\sum_{n\geq
1}\tilde{b}_n\left(\sum_{i=0}^{n}X_1^iX_2^{n-i}\right)\right)
\end{align*}
is given by
\begin{align*}
c_{(m_1,m_2)}=\sum_{|\mu|=m_1+m_2}\frac{(-1)^{l(\mu)-1}(l(\mu)-1)!\tilde{b}_\mu}{|Aut(\mu)|}|S_{\mu}(m_1)|
\end{align*}
where $S_{\mu}(m_1)$ is the set
\begin{align*}
S_\mu(m_1)=\{(i_1,...,i_{l(\mu)})\in
\mathbb{Z}^{l(\mu)}|\sum_{k=1}^{l(\mu)}i_k=m_1,\ \text{where}\ 0\leq
i_k\leq \mu_k, \ \text{for} \ k=1,..,l(\mu)\},
\end{align*}
by this definition,  $S_\mu(m_1)=S_{\mu}(m_2)$.

We write $c_{(m_1,m_2)}=\sum_{l\geq 0}c_{(m_1,m_2),l}a^l$, then the
annulus amplitude is
\begin{align*}
F_{(0,2)}=\sum_{m_1\geq 1,m_2\geq 1}\sum_{l\geq
0}(-1)^{(m_1+m_2)\tau}c_{(m_1,m_2),l}a^{l-\frac{m_1+m_2}{2}}x_1^{m_1}x_{2}^{m_2}.
\end{align*}
If we let $n_{(m_1,m_2),l}=n_{(m_1,m_2),0,l-\frac{m_1+m_2}{2}}$, the
multiple covering formula (\ref{multipecoveringlhole}) for $l=2$
gives
\begin{align*}
F_{(0,2)}=\sum_{m_1\geq 1,m_2\geq 1}\sum_{l\geq
0}\sum_{d|m_1,d|m_2,d|l}\frac{1}{d}n_{(\frac{m_1}{d},\frac{m_2}{d}),\frac{l}{d}}a^{l-\frac{m_1+m_2}{2}}x_1^{m_1}x_{2}^{m_2}
\end{align*}
we have
\begin{align*}
(-1)^{(m_1+m_2)\tau}c_{(m_1,m_2),l}=\sum_{d|m_1,d|m_2,d|l}\frac{1}{d}n_{(\frac{m_1}{d},\frac{m_2}{d}),\frac{l}{d}}.
\end{align*}
so
\begin{align*}
n_{(m_1,m_2),l}=\sum_{d|m_1,d|m_2,d|l}\frac{\mu(d)}{d}(-1)^{\frac{(m_1+m_2)\tau}{d}}c_{(\frac{m_1}{d},\frac{m_2}{d}),\frac{l}{d}}.
\end{align*}
In particular, when $l=\frac{m_1+m_2}{2}$, we only need to consider
the the curve $Y=X(1-Y)^\tau$.  We need the following formula:
\begin{lemma} [Lemma 2.3 of \cite{Zhu0}]
\begin{align}
\ln\left(\frac{Y_{1}(X_1)-Y_{2}(X_2)}{X_{1}-X_{2}}\right)&=\sum_{m_1,m_2
\geq
1}\frac{1}{m_1+m_2}\binom{m_1\tau+m_1-1}{m_1}\binom{m_2\tau+m_2}{m_2}X_1^{m_1}X_2^{m_2}\\\nonumber
&-\tau\left(\ln(1-Y_{1}(X_1))+\ln(1-Y_{2}(X_2))\right).
\end{align}
\end{lemma}
We have, for $m_1,m_2\geq 1$,
\begin{align} \label{annuluscountingunknot}
c_{(m_1,m_2),\frac{m_1+m_2}{2}}(\tau)&=\frac{1}{m_1+m_2}\binom{m_1\tau+m_1-1}{m_1}\binom{m_2\tau+m_2}{m_2}.
\end{align}

For brevity, we let
$n_{(m_1,m_2)}(\tau):=n_{(m_1,m_2),\frac{m_1+m_2}{2}}(\tau)$ which
is defined through formula (\ref{annuluscountingunknot}). Then we
have the following integrality result:
\begin{theorem} \label{theoremgenus0twohole}
For $m_1,m_2\geq 1$, and $\tau\in \mathbb{Z}$,
$n_{(m_1,m_2)}(\tau)\in \mathbb{Z}$.
\end{theorem}
 For nonnegative integer $n$
and prime number $q$, define
\begin{equation}
    f_q(n)=\prod_{i=1,q\nmid i}^n i = \frac{n!}{q^{[n/q]}[n/q]!} \nonumber
\end{equation}
It is obvious that
\begin{equation}
    f_q(q^\alpha k) \equiv f_q(q^{\alpha})^k \equiv (-1)^k \pmod{p^\alpha} \label{luo:eq3-1}
\end{equation}
We introduce the following Lemma first.
\begin{lemma}
    If $p^\beta\mid\mid (a,b), p^\alpha\mid a+b$, then $p^{\alpha-\beta}$ divides
    $$ \binom{a\tau+a-1}{a}\binom{b\tau+b}{b}.$$
    \label{luo:lemma3-1}
\end{lemma}
\begin{proof}
    Power of prime $p$ in $n!$ is
    $ \sum_{k=1}^\infty [\frac{n}{p^k}]. $
    Apply this to the binomial coefficients to find that the power of $p$ in $\binom{a\tau+a-1}{a}\binom{b\tau+b}{b}$ is
    \begin{align*}
        & \sum_{i=1}^{\infty} \left([\frac{a\tau+a-1}{p^i}]+[\frac{b\tau+b}{p^i}]\right)-\left([\frac{a\tau-1}{p^i}]+[\frac{b\tau}{p^i}]\right)-\left([\frac{a}{p^i}]+[\frac{b}{p^i}]\right) \\
        &\geq \sum_{i=1}^{\alpha} \left((\frac{(a+b)(\tau+1)}{p^i}-1)-(\frac{(a+b)\tau}{p^i}-1)\right) - \sum_{i=1}^{\beta}(\frac{a+b}{p^i})-\sum_{i=\beta+1}^{\alpha}(\frac{a+b}{p^i}-1) \\
        &=\alpha-\beta
    \end{align*}
    where we use the fact that for $k\mid m+n+1, k>1$, $[m/k]+[n/k]=(m+n)/k-1$ and for $k\mid m+n, k\nmid m$, $[m/k]+[n/k]=(m+n)/k-1$.
\end{proof}
Now, we can finish the proof of  Theorem \ref{theoremgenus0twohole}:
\begin{proof}
 By definition,
\begin{align}
    n_{(m_1,m_2)}(\tau)& =  \frac{1}{m_1+m_2}\sum_{d\mid m_1,d\mid m_2} \mu(d) (-1)^{(m_1+m_2)(\tau+1)/d} \nonumber \\
                     & \cdot \binom{(m_1\tau+m_1)/d-1}{m_1/d}\binom{(m_2\tau+m_2)/d}{m_2/d} \label{luo:eq1-1}
\end{align}

    Let $p$ be any prime divisor of $m_1+m_2$, $p^\alpha\mid \mid m_1+m_2$. We will prove $p^\alpha$ divides the summation in (\ref{luo:eq1-1}), thus $m_1+m_2$ also divides and $n_{m_1,m_2}$ are integers.

   If $p\nmid m_1$, each summand in (\ref{luo:eq1-1}) corresponds to $p\nmid d$, so $p^\alpha\mid (m_1+m_2)/d$ and $p\nmid m_1/d$. By Lemma~\ref{luo:lemma3-1} applies to $a=m_1/d, b=m_2/d$, $p^{\alpha}$ divides each summand and thus the summation.

   If $p^\beta \mid\mid m_1, \beta\geq 1$, consider two summands in (\ref{luo:eq1-1}) corresponding to $d$ and $pd$ such that $pd\mid (m_1,m_2), \mu(pd)\neq 0$. When $p$ is an odd prime or $\alpha\geq 2$, the sign $(-1)^{(m_1+m_2)(\tau+1)/d}$ and $(-1)^{(m_1+m_2)(\tau+1)/(pd)}$ are equal. When $p=2, \alpha=1$, modulo 2 the sign is irrelevant. Write $a=m_1/d, b=m_2/d$, then $p^\alpha\mid a+b, p^\beta\mid\mid a$.
   \begin{align}
       & \binom{a\tau+a-1}{a}\binom{b\tau+b}{b}-\binom{(a\tau+a)/p-1}{a/p}\binom{(b\tau+b)/p}{b/p} \nonumber \\
       &= \binom{(a\tau+a)/p-1}{a/p}\binom{(b\tau+b)/p}{b/p}\left( \frac{f_p(a\tau+a)f_p(b\tau+b)}{f_p(a\tau)f_p(a)f_p(b\tau)f_p(b)} - 1\right)  \nonumber \\
       &= \binom{(a\tau+a)/p-1}{a/p}\binom{(b\tau+b)/p}{b/p} \frac{f_p(a\tau+a)f_p(b\tau+b)-f_p(a\tau)f_p(a)f_p(b\tau)f_p(b)}{f_p(a\tau)f_p(a)f_p(b\tau)f_p(b)}  \label{luo:eq2-1}
   \end{align}

   The term $\binom{(a\tau+a)/p-1}{a/p}\binom{(b\tau+b)/p}{b/p}$ is divisible by $p^{\alpha-\beta}$ by Lemma~\ref{luo:lemma3-1}. The numerator of the fraction term in (\ref{luo:eq2-1}) is divisible by $p^\beta$ by (\ref{luo:eq3-1}), and the denominator is not divisible by $p$. We proved that $p^\alpha$ divides (\ref{luo:eq2-1}), take summation over $d$, we get that $p^\alpha$ divides the summation in (\ref{luo:eq1-1}).
    This is true for any $p\mid m_1+m_2$, thus $n_{(m_1,m_2)}(\tau)$ is an integer.
\end{proof}

\subsection{Genus g=0 with more  holes}
By formula (\ref{MVGW}), we have
\begin{align*}
K^\tau_{\mu,g,\frac{|\mu|}{2}}&=(-1)^{|\mu|\tau}[\tau(\tau+1)]^{l(\mu)-1}\prod_{i=1}^{l(\mu)}
\frac{\prod_{a=1}^{\mu_{i}-1}(\mu_{i}\tau+a)}{(\mu_{i}-1)!}\int_{\overline{\mathcal{M}}_{g,l(\mu)}}
\frac{\Gamma_{g}(\tau)}{\prod_{i=1}^{l(\mu)}(1-\mu_{i}\psi_{i})}\\\nonumber
&=(-1)^{|\mu|\tau}[\tau(\tau+1)]^{l(\mu)-1}\prod_{i=1}^{l(\mu)}\binom{\mu_i(\tau+1)-1}{\mu_i-1}\sum_{b_i\geq
0}\prod_{i=1}^{l(\mu)}\mu_i^{b_i}\langle\prod_{i=1}^{l(\mu)}\tau_{b_i}\Gamma_g(\tau)
\rangle_{g,l(\mu)}
\end{align*}
When $g=0$ and $l\geq 3$, then $\Gamma_{0}(\tau)=1$ and the Hodge
integrals
\begin{align*}
\langle\tau_{b_1}\cdots\tau_{b_l}\rangle_{0,l}=\binom{l-3}{b_1,..,b_l}.
\end{align*}
Hence, we have
\begin{align} \label{MVGWgenus0}
K^\tau_{\mu,0,\frac{|\mu|}{2}}=(-1)^{|\mu|\tau}[\tau(\tau+1)]^{l(\mu)-1}\prod_{i=1}^{l(\mu)}
\binom{\mu_i(\tau+1)-1}{\mu_i-1}\left(\sum_{i=1}^{l(\mu)}\mu_i\right)^{l(\mu)-3}
\end{align}
And by formula (\ref{multipecoveringgenus0}), we obtain
\begin{align}
n_{\mu,0,\frac{|\mu|}{2}}(\tau)=(-1)^{l(\mu)}\sum_{d|\mu}\mu(d)d^{l(\mu)-1}K^{\tau}_{\frac{\mu}{d},0,\frac{|\mu|}{2d}}
\end{align}
By formula (\ref{MVGWgenus0}), it is clear that
$K^\tau_{\mu,0,\frac{|\mu|}{2}}\in \mathbb{Z}$, and since
$l(\mu)\geq 3$, we obtain
\begin{theorem}
For a partition $\mu$ with $l(\mu)\geq 3$,
\begin{align*}
n_{\mu,0,\frac{|\mu|}{2}}(\tau)\in \mathbb{Z}.
\end{align*}
\end{theorem}

\subsection{Genus $g \geq 1$, with one hole}

\subsubsection{Revist LMOV conjecture for framed knot $\mathcal{K}_\tau$}
We introduce the following notations first. Let $n\in \mathbb{Z}$
and $\lambda,\mu,\nu$ denote the partitions. Let
\begin{align} \label{quantuminteger}
\{n\}_x=x^{\frac{n}{2}}-x^{-\frac{n}{2}}, \
\{\mu\}_{x}=\prod_{i=1}^{l(\mu)}\{\mu_i\}_x.
\end{align}
For brevity, we let $\{n\}=\{n\}_q$ and $\{\mu\}=\{\mu\}_q$. Let
$\mathcal{K}_{\tau}$ be a knot with framing $\tau \in \mathbb{Z}$.
The framed colored HOMFLYPT invariant
$\mathcal{H}(\mathcal{K}_\tau;q,a)$ of $\mathcal{K}_\tau$ is defined
in formula (\ref{framedknotformula}). Let
\begin{align*}
\mathcal{Z}_{\mu}(\mathcal{K}_\tau)=\sum_{\lambda}\chi_{\lambda}(C_\mu)\mathcal{H}_{\lambda}(\mathcal{K}_\tau).
\end{align*}
Then the Chern-Simons partition function is
\begin{align*}
Z_{CS}^{(S^3,\mathcal{K}_\tau)}=\sum_{\lambda\in
\mathcal{P}}\mathcal{H}_\lambda(\mathcal{K}_\tau)s_{\lambda}(x)
=\sum_{\mu\in
\mathcal{P}}\frac{\mathcal{Z}_{\mu}(\mathcal{K}_\tau)}{\mathfrak{z}_\mu}p_{\mu}(x).
\end{align*}
We define $F_{\mu}(\mathcal{K}_\tau)$ though the expansion formula
\begin{align*}
F_{CS}^{(S^3,\mathcal{K}_\tau)}=\log(Z_{CS}^{(S^3,\mathcal{K}_\tau)})=\sum_{\mu\in
\mathcal{P}^+}F_{\mu}(\mathcal{K}_\tau)p_{\mu}(x).
\end{align*}
Then
\begin{align*}
F_{\mu}(\mathcal{K}_\tau)=\sum_{n\geq
1}\sum_{\cup_{i=1}^{n}\nu^i=\mu}\frac{(-1)^{n-1}}{n}\prod_{i=1}^{n}\frac{\mathcal{Z}_{\nu^i}(\mathcal{K}_\tau)}{\mathfrak{z}_{\nu^i}}.
\end{align*}
and we introduce the notation
\begin{align*}
\hat{F}_{\mu}(\mathcal{K}_\tau)=\frac{F_{\mu}(\mathcal{K}_\tau)}{\{\mu\}}.
\end{align*}
\begin{remark}
For two partitions $\nu^1$ and $\nu^2$, the notation $\nu^1\cup
\nu^2$ denotes the new partition by combing all the parts in $\nu^1,
\nu^2$. For example $\mu=(2,2,1)$, then the set of pairs
$(\nu^1,\nu^2)$ such that $\nu^1 \cup \nu^2 =(2,2,1)$ is
\begin{align*}
(\nu^1=(2), \nu^2=(2,1)),\ (\nu^1=(2,1), \nu^2=(2)),\\\nonumber
 \
(\nu^1=(1), \nu^2=(2,2)), \ (\nu^1=(2,2), \nu^2=(1)), \
\end{align*}
\end{remark}
For a rational function $f(q,a)\in \mathbb{Q}(q^{\pm},a^{\pm})$, we
define the adams operator
\begin{align*}
\Psi_{d}(f(q,a))=f(q^d,a^d).
\end{align*}
Then, we have
\begin{align} \label{formulagmu}
\hat{g}_{\mu}(\mathcal{K}_\tau)=\sum_{d|\mu}\frac{\mu(d)}{d}\Psi_{d}(\hat{F}_{\mu/d}(\mathcal{K}_\tau))
\end{align}

The LMOV conjecture for framed knot $\mathcal{K}_\tau$ says:
\begin{conjecture} \label{LMOVframedknot2}
For any partition $\mu$, there exist integers $n_{\mu,g,Q}(\tau)$,
such that
\begin{align*}
\mathfrak{z}_\mu\hat{g}_{\mu}(\mathcal{K}_\tau)=\sum_{g\geq
0}\sum_{Q}n_{\mu,g,Q}(\tau)z^{2g-2}a^Q\in
z^{-2}\mathbb{Z}[z^2,a^{\pm \frac{1}{2}}],
\end{align*}
where $z=q^{\frac{1}{2}}-q^{-\frac{1}{2}}=\{1\}$.
\end{conjecture}

\subsubsection{Framed unknot $U_{\tau}$}
For convenience, we define the function
\begin{align*}
\phi_{\mu,\nu}(x)=\sum_{\lambda}\chi_{\lambda}(C_\mu)\chi_{\lambda}(C_\nu)x^{\kappa_\lambda}.
\end{align*}
By Lemma 5.1  in \cite{CLPZ},  for $d\in \mathbb{Z}_+$, we have
\begin{align*}
\phi_{(d),\nu}(x)=\frac{\{d\nu\}_{x^2}}{\{d\}_{x^2}}.
\end{align*}
By the formula of colored HOMFLYPT invariant for unknot
(\ref{unknotformula}), we obtain
\begin{align*}
\mathcal{Z}_{\mu}(U_\tau)&=\sum_{\lambda}\chi_{\lambda}(C_\mu)\mathcal{H}_{\lambda}(U_\tau)\\\nonumber
&=(-1)^{|\mu|\tau}\sum_{\lambda}\chi_{\lambda}(C_\mu)q^{\frac{\kappa_\lambda
\tau}{2}}\sum_{\nu}\frac{\chi_{\lambda}(C_\nu)}{z_\nu}\frac{\{\nu\}_a}{\{\nu\}}\\\nonumber
&=(-1)^{|\mu|\tau}\sum_{\nu}\frac{1}{\mathfrak{z}_\nu}\phi_{\mu,\nu}(q^\frac{\tau}{2})\frac{\{\nu\}_a}{\{\nu\}}.
\end{align*}
In particular, for $\mu=(m)$, $m\in \mathbb{Z}$, we have
\begin{align*}
\mathcal{Z}_{m}(U_\tau)=(-1)^{m\tau}\sum_{|\nu|=m}\frac{1}{\mathfrak{z}_\nu}\frac{\{m\nu\tau\}}{\{m\tau\}}\frac{\{\nu\}_a}{\{\nu\}}.
\end{align*}
For brevity, we let
$\mathcal{Z}_m(q,a)=\frac{1}{\{m\}}\mathcal{Z}_{m}(U_\tau)=(-1)^{m\tau}\sum_{|\nu|=m}\frac{1}{\mathfrak{z}_\nu}
\frac{\{m\nu\tau\}}{\{m\}\{m\tau\}}\frac{\{\nu\}_a}{\{\nu\}}$ and
$g_m(q,a)=\mathfrak{z}_{(m)}\hat{g}_m(U_\tau)$. Then, by formula
(\ref{formulagmu}), we have
\begin{align} \label{gm}
g_m(q,a)=\sum_{d|m}\mu(d)\mathcal{Z}_{m/d}(q^d,a^d).
\end{align}
Then the integrality of the higher genus with one hole LMOV
invariants is encoded in the following theorem,

\begin{theorem} \label{integralitymainthm}
For any integer $m\geq 1$, there exist integers $n_{m,g,Q}(\tau)$,
such that
\begin{align*}
g_m(q,a)=\sum_{g\geq 0}\sum_{Q}n_{m,g,Q}(\tau)z^{2g-2}a^Q\in
z^{-2}\mathbb{Z}[z^2,a^{\pm \frac{1}{2}}],
\end{align*}
where $z=q^{\frac{1}{2}}-q^{-\frac{1}{2}}=\{1\}$.
\end{theorem}
The proof of Theorem \ref{integralitymainthm} is divided into
several steps. First, we need the following lemmas.
\begin{lemma}
    Suppose $k$ is a positive integer, then the number
    \begin{equation*}
        c_m(k,y)=\sum_{|\lambda|=m} \frac{1}{\mathfrak{z}_\lambda} k^{l(\lambda)}\{\lambda\}_{y^2}
    \end{equation*}
    is equal to the coefficient of $t^m$ in $(\frac{1-t/y}{1-ty})^k$.
\end{lemma}

\begin{proof}
    Suppose the number of $i$'s in the partition $\lambda$ is $a_i, i=1,\cdots$. Then
    \begin{align*}
        c_m(k,y)&=\sum_{\sum ia_i=m} \prod_i \frac{1}{a_i!i^{a_i}} k^{a_i}(y^i-y^{-i})^{a_i} \\
                &=\left[ \prod_{i=1}^{\infty} \left(\sum_{j=0}^\infty t^{ij}\frac{1}{j!i^j} k^j(y^i-y^{-i})^j\right) \right]_{t^m} \\
                &=\left[ \prod_{i=1}^{\infty} \exp(t^ik(y^i-y^{-i})/i) \right]_{t^m} \\
                &=\left[ \exp(k\ln(1-ty)^{-1}+k\ln(1-t/y))\right]_{t^m}\\
                &=\left[ (\frac{1-t/y}{1-ty})^k \right]_{t^m}
    \end{align*}
\end{proof}

\begin{lemma} \label{mmtaulemma}
Let $R=\mathbb{Q}[q^{\pm 1/2},a^{\pm 1/2}]$. Then
\begin{align}  \label{mmtaugm}
 \{m\}\{m\tau\}g_m(q,a)=\sum_{d\mid m}\sum_{|\mu|=m/d}\frac{\mu(d)(-1)^{m\tau/d}}{\mathfrak{z}_\mu}\frac{\{m\mu\tau\}}{\{d\mu\}}\{d\mu\}_a
\end{align}
is divisible by $\{m\tau\}\{m\}/\{1\}^2$ in $R$.
\end{lemma}
\begin{proof}
By the definition (\ref{gm}) of $g_m(q,a)$, we have the formula
(\ref{mmtaugm}). It is clear that
$$\{m\}\{m\tau\}g_m(q,a) \in R.$$
Denote $\Phi_n(q)=\prod_{d\mid n} (q^d-1)^{\mu(n/d)}$ to be the
$n$-th cyclotomic polynomial, which is irreducible over $R$. Then
$q^n-1=\prod_{d\mid n} \Phi_d(q)$, and
\begin{align} \label{cyc}
   \{m\}\{m\tau\}&=q^{-\frac{m+m\tau}{2}} \prod_{{m_1}\mid m} \Phi_{m_1}(q)\prod_{{m_1}\mid m\tau} \Phi_{m_1}(q)\\\nonumber
   &=q^{-\frac{m+m\tau}{2}}\prod_{{m_1}\mid m} \Phi_{m_1}(q)^2\prod_{{m_1}\mid m\tau,
{m_1}\nmid m} \Phi_{m_1}(q)
\end{align}

(i) For $m_1\mid m\tau, m_1\nmid m$, and any $|\mu|=m/d$, at least
one of $d\mu_i$'s are not divisible by $m_1$, thus
$\{m\mu_i\tau\}/\{d\mu_i\}$ is divisible by $\Phi_{m_1}(q)$. So
$\Phi_{m_1}(q)$ divides $\{m\}\{m\tau\}g_m(q,a)$.

(ii) For $m_1\mid m$ and any $|\mu|=m/d$, if not all $d\mu_i$ are
divisible by $m_1$, then at least two of them are not divisible.
Then two of corresponding $\{m\mu_i\tau\}/\{d\mu_i\}$ are divisible
by $\Phi_{m_1}(q)$.

We consider modulo $\{m_1\}^2$ in the ring $R$. It is easy to see,
for $a,b\geq 1 $,
\begin{align*}
\frac{\{abm_1\}}{\{bm_1\}} \equiv
a\left(\frac{q^{m_1/2}+q^{-m_1/2}}{2}\right)^{(a-1)b}
\pmod{\{m_1\}^2}
\end{align*}
We write $x=(q^{m_1/2}+q^{-m_1/2})/2$, then $x^2\equiv 1
\pmod{\{m_1\}^2}$.

Then modulo $\Phi_{m_1}(q)^2$, we have
\begin{align}
    &\{m\}\{m\tau\}g_m(q,a)\nonumber\\
     &\equiv \sum_{d\mid m}\sum_{|\mu|=m/d,m_1\mid d\mu} \frac{\mu(d)(-1)^{m\tau/d}}{\mathfrak{z}_\mu} \frac{\{m\mu\tau\}}{\{d\mu\}}\,\{d\mu\}_a \nonumber\\
                          &\equiv \sum_{d\mid m}\sum_{|\mu|=m/d,m_1\mid d\mu} \frac{\mu(d)(-1)^{m\tau/d}}{\mathfrak{z}_\mu} \left(\frac{m\tau}{d}\right)^{l(\mu)}
                          x^{(m|\mu|\tau-d|\mu|)/m_1} \{d\mu\}_a  \nonumber\\
                          &\equiv \sum_{d\mid m}\sum_{|\lambda|=m/\mathrm{lcm}(d,m_1)} \mu(d)(-1)^{m\tau/d}x^{\frac{m}{m_1}(\frac{m\tau}{d}-1)}
                          \cdot\frac{1}{\mathfrak{z}_\lambda}\left(\frac{m\tau}{\mathrm{lcm}(d,m_1)}\right)^{l(\lambda)} \{\lambda\}_{a^{\mathrm{lcm}(d,m_1)}} \nonumber\\
                          &\equiv \sum_{d\mid m} \mu(d) (-1)^{m\tau/d} x^{\frac{m}{m_1}(\frac{m\tau}{d}-1)} \left[ \left(\frac{1-t^{\mathrm{lcm}(d,m_1)}a^{-\mathrm{lcm}(d,m_1)/2}}{1-t^{\mathrm{lcm}(d,m_1)}a^{\mathrm{lcm}(d,m_1)/2}}\right)^{m\tau/\mathrm{lcm}(d,m_1)} \right]_{t^m}
                          \label{luo:eq44}
\end{align}
\begin{itemize}
    \item For the cases $m_1$ with an odd prime factor $p$, or $p=2$ divides $m_1$ and $4\mid m$, or $p=2$ divides $m_1$ and $2\mid \tau$: Consider those $d$ with $\mu(d)\neq 0$ and $p\nmid d$, we have $\mathrm{lcm}(d,m_1)=\mathrm{lcm}(pd,m_1)$ and parity of $m\tau/d$ equals parity of $m\tau/(pd)$, but $\mu(d)=-\mu(pd)$. Thus two terms in (\ref{luo:eq44}) corresponding to $d$ and $pd$ cancelled.
    \item For the remaining case $2\mid\mid m, m_1=2, 2\nmid\tau$: $\Phi_{m_1}(q)^2=(q^{1/2}+q^{-1/2})^2=2x+2$. Coefficients of $x$ in (\ref{luo:eq44}) equals sum of terms corresponds to odd $d\mid m, \mu(d)\neq 0$, while constant term coefficients equals to sum of terms corresponds to $2d\mid m, \mu(2d)\neq 0$. The coefficients of term for $d$ and $2d$ match, so (\ref{luo:eq44}) is divisible by $x+1$.
\end{itemize}

In summary, we have proved that for $m_1\mid m\tau, m_1\nmid m$,
$\Phi_{m_1}(q)$ divides $\{m\}\{m\tau\}g_m(q,a)$; for $m_1\mid m,
m_1 \neq 1$, $\Phi_{m_1}(q)^2$ divides $\{m\}\{m\tau\}g_m(q,a)$. By
(\ref{cyc}), the lemma is proved.
\end{proof}

\begin{lemma} \label{lemmaprime}
For any integer $m\geq 1$, we have
\begin{align*}
g_m(q,a)\in z^{-2}\mathbb{Q}[z^2,a^{\pm \frac{1}{2}}].
\end{align*}
\end{lemma}
\begin{proof}
By Lemma \ref{mmtaulemma}, we have
\begin{align*}
f(q,a):=z^2g_m(q,a)=\frac{\{1\}^2}{\{m\}\{m\tau\}}\sum_{d\mid
 m}\sum_{|\mu|=m/d}\frac{\mu(d)(-1)^{m\tau/d}}{\mathfrak{z}_\mu}\frac{\{m\mu\tau\}}{\{d\mu\}}\{d\mu\}_a
 \in \mathbb{Q}[q^{\pm \frac{1}{2}},a^{\pm 1}].
\end{align*}
As a function of $q$, it is clear $f(q,a)$ admits $
f(q,a)=f(q^{-1},a). $ Furthermore, for any $d|m$ and $|\mu|=m/d$, we
have
\begin{align*}
   m|\mu|\tau-d|\mu|-m\tau-m\equiv m^2\tau/d-m\tau=m\tau(m/d-1) \equiv 0
   \pmod{2},
\end{align*}
which implies $ f(q,a)=f(-q,a). $ Therefore,
$f(q,a)=z^2g_{m}(q,a)\in \mathbb{Q}[z^2,a^{\pm \frac{1}{2}}]$. The
lemma is proved.
\end{proof}

\begin{lemma} \label{integralitylemma}
For any $\tau\in \mathbb{Z}$, we have
\begin{align} \label{integralityformula}
\{m\}\{m\tau\}\mathcal{Z}_m(q,a)\in \mathbb{Z}[q^{\pm
\frac{1}{2}},a^{\pm \frac{1}{2}}].
\end{align}
\end{lemma}

\begin{proof}
Since
\begin{align*}
(-1)^{m\tau}\{m\}\{m\tau\}\mathcal{Z}_m(q,a)&=\sum_{|\mu|=m}\frac{\{m\tau
\mu\}}{\mathfrak{z}_\mu\{\mu\}}\{\mu\}_a\\\nonumber
&=\sum_{\sum_{j\geq 1}jk_j=m}\frac{\prod_{j\geq 1}(\{m\tau
j\}\{j\}_a)^{k_j}}{\prod_{j\geq 1}j^{k_j}k_j!},
\end{align*}
we construct a generating function
\begin{align} \label{formulafm}
f(x)&=\sum_{n\geq 0}x^n\sum_{\sum_{j\geq 1}jk_j=n}\frac{\prod_{j\geq
1}(\{m\tau j\}\{j\}_a)^{k_j}}{\prod_{j\geq 1}j^{k_j}k_j!}\\\nonumber
&=\sum_{n\geq 0}\sum_{\sum_{j\geq 1}jk_j=n}\frac{\prod_{j\geq
1}(\{m\tau j\}\{j\}_ax^j)^{k_j}}{\prod_{j\geq
1}j^{k_j}k_j!}\\\nonumber &=\exp\left(\sum_{j\geq 1}\frac{\{m\tau
j\}\{j\}_ax^j}{j\{j\}}\right),
\end{align}
Then $
(-1)^{m\tau}\{m\}\{m\tau\}\mathcal{Z}_m(q,a)=\left[f(x)\right]_{x^m}.
$

For $\tau=0$, it is the trivial case.

For $\tau \geq 1$, we use the expansion $\frac{\{m\tau
j\}}{\{j\}}=\sum_{k=0}^{m\tau-1}q^{\frac{j(m\tau-2k-1)}{2}}$, then
\begin{align*}
f(x)&=\exp\left(\sum_{k\geq 0}^{m\tau-1}\sum_{j\geq
1}\left(\frac{(q^{\frac{m\tau-1-2k}{2}}a^{\frac{1}{2}}x)^j}{j}-\frac{(q^{\frac{m\tau-1-2k}{2}}a^{-\frac{1}{2}}x)^j}{j}\right)\right)\\\nonumber
&=\exp\left(\sum_{k\geq
0}^{m\tau-1}\log\frac{1+q^{\frac{m\tau-1-2k}{2}}a^{-\frac{1}{2}}x}{1+q^{\frac{m\tau-1-2k}{2}}a^{\frac{1}{2}}x}\right)\\\nonumber
&=\prod_{k=0}^{m\tau-1}\frac{1+q^{\frac{m\tau-1-2k}{2}}a^{-\frac{1}{2}}x}{1+q^{\frac{m\tau-1-2k}{2}}a^{\frac{1}{2}}x}.
\end{align*}
We introduce the Gaussian binomial coefficients defined by
\begin{align*}
\binom{m}{r}_q=\frac{(1-q^m)(1-q^{m-1})\cdots
(1-q^{m-r+1})}{(1-q)(1-q^2)\cdots (1-q^r)}
\end{align*}
for $r\leq m$, and in particular $\binom{m}{0}_q=1$. The Gaussian
binomial coefficients $\binom{m}{r}_q\in \mathbb{Z}[q]$ (see Chapter
2 of \cite{KS0} for $q$-calculus). There are analogs of the binomial
formula, and of Newton's generalized version of it for negative
integer exponents,
\begin{align*}
\prod_{k=0}^{n-1}(1+q^kt)&=\sum_{k=0}^nq^{\frac{k(k-1)}{2}}\binom{n}{k}_qt^k
\\\nonumber
\prod_{k=0}^{n-1}\frac{1}{(1-q^kt)}&=\sum_{k=0}^\infty
\binom{n+k-1}{k}_qt^k.
\end{align*}

Therefore, the coefficient $\left[f(x)\right]_{x^m}$ of $x^m$ in
$f(x)$ is given by
\begin{align*}
\sum_{j+k=m}(-1)^{k}q^{\frac{j(j-1)-(m\tau-1)m}{2}}a^{\frac{k-j}{2}}\binom{m\tau}{j}_q\binom{m\tau+k-1}{k}_q,
\end{align*}
which lies in the ring $\mathbb{Z}[q^{\pm \frac{1}{2}},a^{\pm
\frac{1}{2}}]$ by the integrality of Gaussian binomial.

For the case $\tau \leq -1$, we write $\{m\tau j\}=-\{-m\tau j\}$ in
the formula (\ref{formulafm}), then the similar computations give
the formula (\ref{integralityformula}).
\end{proof}

\begin{remark}
In fact, by using the Theorem 3.2 in \cite{CLPZ}, we have the
following refiner integrality structure:
\begin{align*}
\{m\}^2\mathcal{Z}_{m}(q,a)\in \mathbb{Z}[z^2,a^{\pm \frac{1}{2}}].
\end{align*}
Together with Lemma \ref{lemmaprime}, it can also be used to
complete the proof of Theorem \ref{integralitymainthm}.
\end{remark}

Now, we can finish the proof of Theorem \ref{integralitymainthm} as
follow:
\begin{proof}
Lemma \ref{lemmaprime}  implies that there exist
 rational numbers $n_{m,g,Q}(\tau)$, such that
\begin{align*}
z^2g_{m}(q,a)=\sum_{g \geq 0}\sum_{Q}n_{m,g,Q}(\tau)z^{2g}a^{Q}\in
\mathbb{Q}[z^2,a^{\pm \frac{1}{2}}].
\end{align*}
So we only need to show $n_{m,g,Q}(\tau)$ are integers. By lemma
\ref{integralitylemma} and the formula (\ref{gm}) for $g_{m}(q,a)$,
we have
\begin{align*}
\{m\}\{m\tau\}z^2g_{m}(q,a)\in \mathbb{Z}[q^{\pm \frac{1}{2}},a^{\pm
\frac{1}{2}}],
\end{align*}
which is equivalent to
\begin{align*}
(q^{\frac{m}{2}}-q^{-\frac{m}{2}})(q^{\frac{m\tau}{2}}-q^{-\frac{m\tau}{2}})\sum_{g
\geq 0}\sum_{Q}n_{m,g,Q}(\tau)(q^{1/2}-q^{-1/2})^{2g}a^{Q}\in
\mathbb{Z}[q^{\pm 1},a^{\pm \frac{1}{2}}].
\end{align*}
So it is easy to get the contradiction if we assume there exists
$n_{m,g,Q}(\tau)$ which is not integer.
\end{proof}

\section{An open string GW/DT correspondence}
\subsection{Reduced open string partition function of $(\mathbb{C}^3,D_\tau)$}
\begin{definition}
We define the reduced open string partition function  of $(X,D)$ as
\begin{align} \label{redopenstring}
\tilde{Z}_{str}^{(X,D)}(g_s,a,x)=Z_{str}^{(X,D)}(g_s,a,\mathbf{x}=(x,0,0,...)).
\end{align}
similarly, the reduced Chern-Simons partition function of
$(S^3,\mathcal{K})$
\begin{align} \label{redCS}
\tilde{Z}_{CS}^{(S^3,\mathcal{K})}(q,a,x)&=Z_{CS}^{(S^3,\mathcal{K})}(q,a,\mathbf{x}=(x,0,0,...))\\\nonumber
&=\sum_{n\geq 0}\mathcal{H}_n(\mathcal{K};q,a)x^n.
\end{align}
since by the definition of Schur function
$s_{\lambda}(\mathbf{x}=(x,0,0,..))=x^{|\lambda|}$ which is nozero
only when $\lambda$ is an one row partition.
\end{definition}

Now we consider the trivial Calabi-Yau 3-fold $\mathbb{C}^3$ with
one AV-brane in framing $\tau$ which can be viewed as the limit case
of the resolved conifold geometry $(\hat{X},D_\tau)$ by study their
toric diagrams. In fact, the open string free energy on
$(\mathbb{C}^3,D_\tau)$ is given by
\begin{align*}
F^{(\mathbb{C}^3,D_\tau)}_{str}(g_s,\mathbf{x})=-\sum_{g\geq0,\mu}\frac{\sqrt{-1}^{l(\mu)}}{|Aut(\mu)|}g_s^{2g-2+l(\mu)}K_{\mu,g,\frac{|\mu|}{2}}^{
\tau}p_{\mu}(\mathbf{x})
\end{align*}
where $K_{\mu,g,\frac{|\mu|}{2}}^{ \tau}$ is the triple Hodge
integral given by formula (\ref{MVGW}).

 We define
\begin{align*}
H_{\lambda}(q)=[\mathcal{H}_{\lambda}(U_\tau,q,a)]_{a^{\frac{|\lambda|}{2}}}=(-1)^{|\lambda|\tau}q^{\frac{\kappa_\lambda
\tau}{2}}\sum_{\mu}\frac{\chi_{\lambda}(C_\mu)}{\mathfrak{z}_\mu}\frac{1}{\{\mu\}}.
\end{align*}
In particular
\begin{align*}
H_{n}(q)=(-1)^{n(\tau-1)}q^{\frac{n(n-1)}{2}\tau+\frac{n^2}{2}}\frac{1}{(1-q)(1-q^2)\cdots(1-q^n)}.
\end{align*}

By large N duality (\ref{LargeNdualiyofunknot}), we have
\begin{align*}
Z^{(\mathbb{C}^3,D_\tau)}_{str}(g_s,\mathbf{x})=\exp\left(F^{(\mathbb{C}^3,D_\tau)}_{str}(g_s,\mathbf{x})\right)=\sum_{\lambda\in
\mathcal{P}}H_{\lambda}(q) s_{\lambda}(\mathbf{x}).
\end{align*}
Then, the reduced open string partition function of
$(\mathbb{C}^3,D_\tau)$ is given by
\begin{align}  \label{partitionfunctionC3}
\tilde{Z}^{(\mathbb{C}^3,D_\tau)}_{str}(g_s,x)=Z^{(\mathbb{C}^3,D_\tau)}_{str}(g_s,\mathbf{x}=(x,0,0,..))=\sum_{n\geq
0}H_{n}(q)x^n.
\end{align}
For brevity, we let
\begin{align*}
Z_\tau(q,x)&=\tilde{Z}^{(\mathbb{C}^3,D_\tau)}_{str}(g_s,x)=\sum_{n\geq
0}H_{n}(q)x^n\\\nonumber &=\sum_{n\geq
0}\frac{(-1)^{n(\tau-1)}q^{\frac{n(n-1)}{2}\tau+\frac{n^2}{2}}}{(1-q)(1-q^2)\cdots(1-q^n)}x^n
\end{align*}

In fact, the LMOV conjecture \ref{LMOVframedknot} provides a
factorization for the partition function $Z_{\tau}(q,a)$, which will
be showed as follow.  We first formulate the LMOV conjecture for the
general reduced partition function $\tilde{Z}$ as in formulas
(\ref{redopenstring}) and (\ref{redCS}). We take the reduced
Chern-Simon partition function for example.
\begin{align*}
\tilde{Z}_{CS}^{(S^3,\mathcal{K}_\tau)}(q,a,x)=\sum_{n\geq
0}\mathcal{H}_n(\mathcal{K}_\tau;q,a)x^n.
\end{align*}
By LMOV conjecture for $\mathcal{K}_\tau$ (Conjecture
\ref{LMOVframedknot}), there exist functions
$f_{m}(\mathcal{K}_\tau;q,a)$ such that
\begin{align*}
\tilde{Z}_{CS}^{(S^3,\mathcal{K}_\tau)}(q,a,x)=\exp\left(\sum_{m\geq
1}\sum_{d\geq
1}\frac{1}{d}f_{m}(\mathcal{K}_\tau;q^d,a^d)x^{dm}\right)
\end{align*}
One can compute $f_{m}(\mathcal{K}_\tau;q,a)$ explicitly,  for
example
\begin{align*}
f_{1}(\mathcal{K}_\tau;q,a)=\mathcal{H}_{1}(\mathcal{K}_\tau).
\end{align*}
\begin{align*}
f_{2}(\mathcal{K}_\tau;q,a)=\mathcal{H}_{2}(\mathcal{K}_\tau)-\frac{1}{2}\mathcal{H}_{1}(\mathcal{K}_\tau)^2
-\frac{1}{2}\Psi_2(\mathcal{H}_{1}(\mathcal{K}_\tau)).
\end{align*}
Then reduced LMOV conjecture asserts the following weak form of the
integrality which was first proposed in \cite{OV}.
\begin{conjecture}[Reduced LMOV conjecture for $\mathcal{K}_\tau$] \label{redLMOV}
There exist integers $N_{m,i,k}(\tau)$, and only finitely many
$N_{m,i,k}(\tau)$ are nonzero for any fixed $m\geq 1$. Such that
\begin{align*}
f_m(\mathcal{K}_\tau;q,a)=-\sum_{i,k\in
\mathbb{Z}}\frac{N_{m,i,k}(\tau)a^{\frac{i}{2}}q^{\frac{k+1}{2}}}{1-q}
\end{align*}
\end{conjecture}
The integers $N_{m,i,k}(\tau)$ are called the Ooguri-Vafa invariants
which were first studied by Ooguri and Vafa in \cite{OV}.

Therefore, by Conjecture \ref{redLMOV}, we have
\begin{align*}
\tilde{Z}_{CS}^{(S^3,\mathcal{K}_\tau)}(q,a,x)&=\exp\left(-\sum_{m\geq
1}\sum_{i,k}N_{m,i,k}(\tau)\sum_{l\geq 0}\sum_{d\geq
1}\frac{1}{d}(a^{\frac{i}{2}}q^{\frac{k+1}{2}+l}x^m)^d
\right)\\\nonumber &=\exp\left(\sum_{m\geq
1}\sum_{i,k}N_{m,i,k}(\tau)\sum_{l\geq
0}\log\left(1-a^{\frac{i}{2}}q^{\frac{k+1}{2}+l}x^m\right)\right)\\\nonumber
&=\prod_{m\geq 1}\prod_{i,k\in \mathbb{Z}}\prod_{l\geq
0}\left(1-a^{\frac{i}{2}}q^{\frac{k+1}{2}+l}x^m\right)^{N_{m,i,k}(\tau)}
\end{align*}

Now, we consider the reduced open string partition function
$Z_{\tau}(q,x)=\tilde{Z}_{str}^{(\mathbb{C}^3,D_\tau)}(g_s,x)$,
 then the
corresponding reduced LMOV conjecture assert that:
\begin{conjecture}[Reduced LMOV conjecture for
$(\mathbb{C}^3,D_\tau)$] \label{redLMOVforC3} There exist integers
$N_{m,k}(\tau)$, and only finitely many $N_{m,k}(\tau)$ are nonzero
for any fixed $m\geq 1$. Such that
\begin{align*}
Z_{\tau}(q,x)=\prod_{m\geq 1}\prod_{k\in \mathbb{Z}}\prod_{l\geq
0}\left(1-q^{\frac{k+1}{2}+l}x^m\right)^{N_{m,k}(\tau)}.
\end{align*}
\end{conjecture}

\subsection{Hilbert-Poincar\'e series of the
Cohomological Hall algebra of the $m$-loop quiver}

We first review the definition and the main results of Cohomological
Hall algebra \cite{KS} for the $m$-loop quiver, $m\in \mathbb{N}$.
Here we mainly following the expositions in \cite{Re}(i.e. Section 4
in \cite{Re}).

Fix a nonnegative integer $m\geq 1$. For a complex vector space $V$,
we denote by $E_V=End(V)^m$ the space of $m$-tuples of endomorphisms
of $V$. Then the group $G_V=GL(V)$ acts on $E_V$ by simultaneous
conjugation.  We study the equivariant cohomology with rational
coefficient $H^*_{G_V}(E_V)$. For two complex vector spaces $V$ and
$W$, Kontsevich and Soibelman \cite{KS} constructed a map:
\begin{align*}
H^*_{G_V}(E_V)\otimes H^*_{G_W}(E_W)\rightarrow
H^{*+\text{shift}}_{G_{V\oplus W}}(E_{V\oplus W}).
\end{align*}
They proved such maps induce an associative unital
$\mathbb{Q}$-algebra structure on $\mathcal{H}=\oplus_{n\geq
0}H^{*}_{G_{\mathbb{C}^n}}(E_{\mathbb{C}^n})$, which is
$\mathbb{N}\times\mathbb{Z}$-bigraded if
$H^{k}_{G_{\mathbb{C}^n}}(E_{\mathbb{C}^n})$ is placed in bidegree
$(n,(m-1)\binom{n}{2}-\frac{k}{2})$. This algebra $\mathcal{H}$ is
called the Cohomological Hall algebra of the $m$-loop quiver in
\cite{KS}. We define the Hilbert-Poincar\'e series of $\mathcal{H}$
as following:
\begin{align*}
P_{m}(q,t)=\sum_{n\geq 0}\sum_{k\in
\mathbb{Z}}\dim_{\mathbb{Q}}\mathcal{H}_{n,k}q^{-k}t^n.
\end{align*}
Please note that we use the parameter $q^{-1}$ instead of $q$ in
Section 4 of \cite{Re}.
\begin{proposition}[Lemma 4.2 \cite{Re}]
The series
\begin{align*}
P_{m}(q,t)=\sum_{n\geq
0}\frac{q^{-(m-1)\frac{n(n-1)}{2}}}{(1-q)(1-q^2)\cdots(1-q^n)}t^n\in
\mathbb{Q}(q)[[t]].  \label{HPseries}
\end{align*}
\end{proposition}
The main property of $P_{m}(q,t)$ is the following factorization
formula.
\begin{theorem}[Conjecture 3.3 \cite{Re} or Theorem 2.3 \cite{KS}]
\label{HPfactorization} There exists a product expansion
\begin{align*}
P_{m}(q,(-1)^{m-1}t)=\prod_{n\geq 1}\prod_{k\geq 0}\prod_{l\geq
0}(1-q^{l-k}t^n)^{-(-1)^{(m-1)n}c_{n,k}}
\end{align*}
for nonnegative integers $c_{n,k}$, such that only finitely many
$c_{n,k}$ are nonzero for any fixed $n$.
\end{theorem}

Let $DT_n^{(m)}(q)=\sum_{k\geq 0}c_{n,k}q^k$ which is called the
quantum Donaldson-Thomas invariant in \cite{Re}. The above theorem
implies that $DT_{n}^{(m)}(q)$ is a polynomial with nonnegative
coefficients. The explicit formula for $DT_n^{(m)}(q)$ was given in
\cite{Re}.

\subsection{The correspondence}
One can write the partition function $Z_{\tau}(q,a)$ in the
following form:
\begin{align*}
Z_{-\tau}(q,x)=\sum_{n\geq
0}\frac{q^{-(\tau-1)\frac{n(n-1)}{2}}}{(1-q)(1-q^2)\cdots
(1-q^n)}((-1)^{\tau-1}xq^{\frac{1}{2}})^n.
\end{align*}
By comparing with the Hilbert-Poincar\'e series $P_m(q,t)$
(\ref{HPseries}), we obtain
\begin{theorem} \label{GWDTcorrespondence}
For $\tau\leq -1$ (i.e. $-\tau\geq 1$), we have
\begin{align*}
Z_{\tau}(q,x)=P_{-\tau}(q,(-1)^{\tau-1}xq^{\frac{1}{2}}).
\end{align*}
\end{theorem}
Theorem \ref{GWDTcorrespondence} can be viewed as an open string
GW/DT correspondence, we refer to \cite{MOOP} for a discussion of
the GW/DT correspondence for toric 3-folds.

Theorem \ref{HPfactorization} implies that, for $\tau\leq -1$, the
reduced open string partition function $Z_{\tau}(q,x)$ on
$(\mathbb{C}^3,D_{\tau})$ carries the product factorization:
\begin{align*}
Z_{\tau}(q,x)=\prod_{n\geq 1}\prod_{k\geq 0}\prod_{l\geq
0}(1-q^{\frac{n}{2}+l-k}x^n)^{-(-1)^{(\tau-1)n}c_{n,k}}.
\end{align*}
Comparing with the Conjecture \ref{redLMOVforC3}, it provides the
correspondence of the Ooguri-Vafa invariants $N_{m,k}(\tau)$ and the
Donaldson-Thomas invariants $c_{n,k}$ for $\tau\leq -1$.

\begin{remark}
For the simplicity of the discussion of the LMOV invariants,
Garoufalidis, Kucharski and Sulkowski \cite{GKS} introduced the
notion of extremal LMOV invariant. In fact, the Ooguri-Vafa
invariants (or weak LMOV invariants) $N_{m,k}(\tau)$ for
$(\mathbb{C}^3,D_\tau)$ in Conjecture \ref{redLMOVforC3} are the
extremal LMOV variants in the sense of \cite{GKS} for framed unknot
$U_\tau$. The relationship of the extremal LMOV invariants and the
work of Reineke \cite{Re} was extensively studied in the recent
paper \cite{KS1}.
\end{remark}

\section{Appendix}
 In \cite{GKS}, Garoufalidis, Kucharski
and Sulkowski obtained the following extremal BPS invariants of
twist knots:
\begin{equation}
    b_{K_p,r}^{-}=-\frac{1}{r^2}\sum_{d\mid r}\mu(\frac{r}{d}) \binom{3d-1}{d-1}, \qquad
    b_{K_p,r}^{+}=\frac{1}{r^2}\sum_{d\mid r}\mu(\frac{r}{d}) \binom{(2|p|+1)d-1}{d-1}
    \label{luo:eq1}
\end{equation}
for $p\leq -1$ and
\begin{equation}
    b_{K_p,r}^{-}=-\frac{1}{r^2}\sum_{d\mid r}\mu(\frac{r}{d}) (-1)^{d+1}\binom{2d-1}{d-1}, \qquad
    b_{K_p,r}^{+}=\frac{1}{r^2}\sum_{d\mid r}\mu(\frac{r}{d}) (-1)^{d} \binom{(2p+2)d-1}{d-1}
    \label{luo:eq22}
\end{equation}
for $p\geq 2$. See the formulas (1.4) and (1.5) in \cite{GKS}.

 In fact, in their later work \cite{KS1},  Kucharski and Sulkowski found the work
of Reineke \cite{Re} can be used to interpret the integrality of
above BPS invariants $b_{K_p,r}^{-}$ and $b_{K_p,r}^{+}$. But in
this appendix, we provide a direct proof of the integrality of the
BPS invariants $b_{K_p,r}^{-}$ and $b_{K_p,r}^{+}$ by the same
method used in the proofs of Theorems \ref{integralitythmonehole},
\ref{theoremgenus0twohole}.

\begin{theorem} \label{appendixthm}
    $b_{K_p,r}^{-}$ and $b_{K_p,r}^{+}$ given in formulas (\ref{luo:eq1}) and (\ref{luo:eq22}) are integers.
\end{theorem}

For nonnegative integer $n$ and prime number $q$, define
\begin{equation}
    f_q(n)=\prod_{i=1,q\nmid i}^n i = \frac{n!}{q^{[n/q]}[n/q]!} \nonumber
\end{equation}

\begin{lemma}[=Lemma \ref{functionfp}]
    For odd prime numbers $q$ and $\alpha\geq 1$ or for $q=2$, $\alpha\geq 2$, we have $q^{2\alpha}\mid f_q(q^{\alpha} n)-f_q(q^{\alpha})^n$.
    For $q=2, \alpha=1$, $f_2(2n)\equiv (-1)^{[n/2]}\pmod{4}$ \label{luo:lemma1}
\end{lemma}

\begin{lemma}
    For prime number $q$ and $m=q^\alpha a, q\nmid a, \alpha\geq 1, k\geq 1$, $q^{2\alpha}$ divides
    \begin{equation}
        (-1)^{(k+1)m}\binom{km-1}{m-1}-(-1)^{(k+1)m/q}\binom{km/q-1}{m/q-1}. \nonumber
    \end{equation}
    \label{luo:lemma2}
\end{lemma}
\begin{proof}
    \begin{align}
        &(-1)^{(k+1)m}\binom{km-1}{m-1}-(-1)^{(k+1)m/q}\binom{km/q-1}{m/q-1} \nonumber \\
       &= (-1)^{(k+1)m}\binom{km/q-1}{m/q-1}\left( \frac{f_q(km)}{f_q((k-1)m)f_q(m)} -(-1)^{(k+1)(m-m/q)}\right) \label{luo:eq2}
    \end{align}
    For $q>2$ or $q=2, \alpha>1$, then $m-m/q$ is even, thus (\ref{luo:eq2}) is divisible by $q^{2\alpha}$ by Lemma~\ref{luo:lemma1}. For $q=2, \alpha=1$, (\ref{luo:eq2}) is divisible by 4 if
 \begin{equation}
        [\frac{km}{4}]+[\frac{(k-1)m}{4}]+[\frac{m}{4}] -(k+1)(m-\frac{m}{2})\equiv 0, \pmod{2} \nonumber
    \end{equation}
    which depends only on $k\pmod{2}$, verify for $k\in \{1,2\}$ to get the results.
\end{proof}

Now, we can finish the proof of Theorem \ref{appendixthm}:
\begin{proof}
    For each prime divisor $q$ of $r$, in the summation in (\ref{luo:eq1}) over $d|r$, pairing the terms with nonzero $\mu(\frac{r}{d})$ and $\mu(\frac{r}{qd})$. Sum of two terms of each pair is divisible by $q^{2\alpha}$ by Lemma~\ref{luo:lemma2}. This is true for all prime divisors of $r$, thus $b_{K_p,r}^{-}$ and $b_{K_p,r}^{+}$ are integers.
\end{proof}

\vskip 30pt

$$ \ \ \ \ $$

\end{document}